\documentclass[reqno,10pt,a4paper]{amsart}

\usepackage[margin=0.8in]{geometry}
\usepackage{parskip}
\usepackage{amsmath}
\usepackage{lmodern}
\usepackage{amsfonts}
\usepackage{amssymb}
\usepackage{amsthm}
\usepackage{xspace}
\usepackage{array}
\usepackage{url}
\usepackage[bookmarks=false]{hyperref}
\usepackage[dvipsnames]{xcolor}
\usepackage{mathrsfs}
\usepackage{mathtools}
\usepackage{scalerel}
\usepackage{enumitem}
\usepackage{titlesec}
\usepackage{titletoc}
\titleformat{\section}
  {\normalfont\fontsize{12}{15}\bfseries}{\thesection.}{1em}{}
\titleformat{\subsection}
  {\normalfont\fontsize{12}{15}\itshape}{\thesubsection.}{1em}{}
\titleformat{\subsubsection}
  {\normalfont\fontsize{12}{15}\itshape}{\thesubsubsection.}{1em}{}

\hypersetup
    {
        colorlinks    = true,
        linkcolor     = blue,     
        urlcolor      = SeaGreen,
    	citecolor     = JungleGreen,
    	bookmarksopen = false,
    }
\theoremstyle{plain}
\newtheorem{theorem}{Theorem}
\newtheorem{lemma}[theorem]{Lemma}
\newtheorem{proposition}[theorem]{Proposition}
\newtheorem{corollary}[theorem]{Corollary}

\theoremstyle{remark}

\numberwithin{equation}{section}
\numberwithin{theorem}{section}

\newcommand{\ie}{i.e.\@\xspace}

\newcommand{\R}{\mathbb{R}}
\newcommand{\N}{\mathbb{N}}

\newcommand{\tr}{\operatorname{tr}}
\renewcommand{\S}{\mathbb{S}}

\DeclareMathOperator{\vol}{\mathrm{vol}}
\renewcommand{\div}{\operatorname{div}}

\DeclareMathOperator{\hess}{Hess}
\DeclareMathOperator{\tracelessHess}{\mathring{Hess}}

\renewcommand{\sec}{\operatorname{sec}}
\DeclareMathOperator{\scal}{scal}
\DeclareMathOperator{\ric}{Ric}

\DeclareMathOperator{\sn}{sn}
\newcommand{\snk}{\sn_k}

\DeclareMathOperator{\arcsn}{arcsn}
\newcommand{\asnk}{\arcsn_k}

\DeclareMathOperator{\cs}{cs}
\newcommand{\csk}{\cs_k}

\DeclareMathOperator{\tn}{tn}
\newcommand{\tnk}{\tn_k}

\DeclareMathOperator{\arctn}{arctn}
\newcommand{\atnk}{\arctn_k}

\newcommand{\kn}{\mathbin{\bigcirc\mspace{-15mu}\wedge\mspace{3mu}}}

\newcommand*{\dd}{\mathop{}\!\mathrm{d}}

\newcommand{\dist}{\mathrm{dist}}

\titleformat{\section}{\normalfont\fontsize{12}{15}\bfseries}{\thesection.}{1em}{}

\titleformat{\subsection}[runin]{\bfseries}{\thesubsection.}{0.5em}{}
\titlespacing*{\subsection}{0pt}{2ex}{1em}

\titleformat{\chapter}[display]
{\Large}
{\filright\Large\thechapter}
{0ex}
{\titlerule[1pt]\filright}

\titlecontents{section} % set formatting for \section -
[1.5em]                 % adjust left margin
{\large\bfseries}  % font formatting
{\contentslabel{1.5em}} % section label and offset
{\hspace*{0em}}
{\normalfont\titlerule*[0.7pc]{.}\contentspage}

\titlecontents{subsection} % set formatting for \subsection -
[3.5em]                    % adjust left margin
{\large}              % font formatting
{\contentslabel{2em}}      % section label and offset
{\hspace*{0.3em}}
{\titlerule*[0.7pc]{.}\contentspage}

\titlecontents{subsubsection} % set formatting for \subsubsection -
[6.5em]                       % adjust left margin
{\large}                 % font formatting
{\contentslabel{3em}}         % section label and offset
{\hspace*{0.3em}}
{\titlerule*[0.7pc]{.}\contentspage}

\begin{document}

\title[Sharp Hessian inequalities and monotonicity on positively curved Einstein manifolds]{Sharp Hessian inequalities and monotonicity for Green's functions of positively curved Einstein manifolds}

\author[Cosmin Manea \and Jacob Reznikov]{Cosmin Manea \and Jacob Reznikov}
\address{Department of Mathematics, Massachusetts Institute of Technology, 77 Massachusetts Avenue, Cambridge, MA, USA}
\email{cosmanea@mit.edu \and jacob\_oc@mit.edu}

\begin{abstract}
    We prove a sharp Hessian inequality for the natural comparison Green's function of an Einstein manifold with positive sectional curvature, which may be viewed as a Green's function analogue of the Hessian comparison theorem and as a positively-curved, elliptic counterpart to the matrix Li-Yau-Hamilton inequality for the heat equation.
    We also show that this inequality is closely related to a family of monotonicity formulae, extending previous results of Park to the setting of positive curvature.
    We further present several geometric applications, such as a sharp comparison inequality for the sum of two distinct Green's distance functions.
\end{abstract}

\maketitle

\section{Introduction} \label{sec: intro}

We prove a sharp Hessian inequality for the natural comparison Green's function of an Einstein manifold with positive sectional curvature, thus providing a sharp, positively-curved extension of {\cite{ParkHessianEstimate}} and a sharp, elliptic analogue of the matrix Li-Yau-Hamilton inequality for the heat equation (cf.~{\cite{HamiltonLiYau}}).
Our Green's function is the fundamental solution for a Schr\"{o}dinger operator with constant, curvature-dependent potential, which acts as a comparison operator for the conformal Laplacian.
We also present a two-parameter family of sharp monotonic quantities closely related to this Hessian inequality (extending those from {\cite{ParkHessianMonotonicity}}), and show geometric applications of our results, such as a comparison inequality for the sum of two distinct regularized Green's distance functions.
We further show that the analogous sharp estimate may generally fail for non-parabolic manifolds with non-negative sectional curvature (which is a setting similar to that from {\cite{ParkHessianEstimate}}).

\subsection{The sharp Hessian inequality.} \label{intro-subsec: sharp hessian inequality}
For Einstein manifolds with positive sectional curvature, our sharp Hessian inequality takes the following form.

\begin{theorem} \label{intro-th: sharp hessian estimate}
Let $(M, g)$ be a closed, connected, Einstein manifold of dimension $n \geq 3$, with sectional curvature bounded below by $k > 0$, and let $G$ be Green's function{\footnote{Our convention is that, if $L$ is a second-order differential operator on a manifold of dimension $n \geq 3$, then every fundamental solution $G$ for $L$ with singularity at $p$ is normalized to solve the equation $LG = (n - 2) \omega_{n - 1} \delta_p$, where $\omega_{n - 1}$ is the volume of the unit sphere in $\R^n$ and $\delta_p$ is the Dirac distribution at $p$.}} for the operator $-\Delta + n(n - 2)k/4$, with singularity at a fixed point of $M$.

Then, for $s := G^{\frac{1}{2 - n}}$, away from the singularity of $G$, we have that
\begin{equation} \label{intro-eq: sharp hessian estimate}
    \hess{s^2} + ks^2g \leq 2g
\end{equation}
in the sense of quadratic forms.
Moreover, equality holds in \eqref{intro-eq: sharp hessian estimate} at \emph{some} non-zero vector if and only if $(M, g)$ is isometric to the round $n$-sphere of constant curvature equal to $k$.
\end{theorem}

To explain Theorem \ref{intro-th: sharp hessian estimate}, let us start by considering the case of the round $n$-sphere of constant curvature $k$.
If $\Delta$ is the non-positive Laplacian, the operator $-\Delta + n(n - 2)k/4$  is the conformal Laplacian, and its Green's function may be explicitly computed as $G = (2 \snk(\dist(p, \cdot)/2))^{2 - n}$, where $\snk$ is the curvature-adapted sine function and $p$ is the singularity of $G$ (cf.~{\cite[Lemma 2.1]{ManeaSGE}}).
In particular, $s = 2 \snk (\dist(p, \cdot)/2)$, and one may then check that $\hess{s^2} + k s^2 g = 2 g$.

On a general closed, Einstein manifold $(M, g)$ of dimension $n \geq 3$ with sectional curvature bounded below by some $k > 0$, the operator $-\Delta + n(n - 2)k/4$ acts as a comparison operator for the conformal Laplacian from the point of view of the model space of constant curvature $k$, and it is equal to it when the curvature is constant.
Under only the weaker assumption that $\ric \geq (n - 1)kg$, the first author previously proved (cf.~{\cite{ManeaSGE}}) a sharp gradient estimate for its Green's function, namely the inequality
\begin{equation} \label{intro-eq: sharp gradient estimate}
    \|\nabla s\|^2 + \frac{k}{4} s^2 \leq 1
\end{equation}
with equality holding at \emph{some} point only when $(M, g)$ is isometric to the round $n$-sphere of curvature $k$.
This extended previous results of Colding (cf.~{\cite{ColdingNewMonotonicityFormulas}}) from the setting of non-parabolic manifolds with non-negative Ricci curvature to the positively-curved setting, and \eqref{intro-eq: sharp gradient estimate} recovers Colding's results in the limit as $k \downarrow 0$.
Various monotonic quantities of Colding (cf.~{\cite{ColdingNewMonotonicityFormulas}}) and Colding--Minicozzi (cf.~{\cite{ColdingMinicozziMonotonicityFormulas}}) have also been extended by the first author to the setting from \eqref{intro-eq: sharp gradient estimate}, as well as other generalizations of the level set approach of Stern {\cite{SternLevelSetScalarCurvature}} and Munteanu--Wang (cf.~{\cite{MunteanuWangComparisonPaper}}) to three-manifolds with positive scalar curvature (cf.~{\cite{ManeaPSC}}).

In this setting, $s$ may be viewed as a regularization of the curvature-adapted sine of the distance function to the singularity of $G$, which also satisfies an elliptic equation.
The gradient estimate \eqref{intro-eq: sharp gradient estimate} is equivalent to the inequality $\Delta s^2 + n k s^2 \leq 2n$, which may thus be seen as a Green's function analogue of the Laplacian comparison theorem that also has a global rigidity statement.
In this sense, Theorem \ref{intro-th: sharp hessian estimate} may also be understood as a Green's function analogue of the Hessian comparison theorem in positive sectional curvature, which is also globally rigid; moreover, tracing \eqref{intro-eq: sharp hessian estimate} immediately gives us \eqref{intro-eq: sharp gradient estimate}.
At the same time, this Hessian bound is also a positively-curved, elliptic counterpart to Hamilton's matrix Li-Yau inequality for positive solutions to the heat equation (cf.~{\cite{HamiltonLiYau}}) and a sharp, positively-curved extension of {\cite[Theorem 1.1]{ParkHessianEstimate}}.

This latter point of view allows us to explain the additional requirement of having an Einstein metric as the hypothesis of Theorem \ref{intro-th: sharp hessian estimate}.
Both Hamilton's and Park's results assume non-negative sectional curvature and parallel Ricci curvature; this last assumption is needed to simplify the action of the Lichnerowicz Laplacian on symmetric, covariant two-tensors.
In the spirit of {\cite{HamiltonLiYau}} and {\cite{ParkHessianEstimate}}, the setting of Theorem \ref{intro-th: sharp hessian estimate} has the same assumption of parallel Ricci curvature, but the de Rham decomposition theorem implies that this is equivalent to the manifold being Einstein.
In this sense, we do not expect Theorem \ref{intro-th: sharp hessian estimate} to hold without this additional assumption.

Whilst the Einstein condition in Theorem \ref{intro-th: sharp hessian estimate} is rather restrictive and essential from the analytic point of view of the proof of such results, there are many examples of Einstein manifolds with positive sectional curvature.
Besides the rank one locally symmetric spaces of compact type (cf.~for instance {\cite[Chapter 3]{BesseClosedGeodesics}} for a description of their geometry), {\cite{ZillerCROSSMetrics}} provides other non-locally symmetric examples (and a complete classification) of homogeneous Einstein metrics with positive sectional curvature on high-dimensional spheres and projective spaces.
In {\cite{BergerSpace}}, Berger showed that, up to scaling, there exists a unique $\mathrm{SO}(5)$-invariant Einstein metric with positive curvature on $\mathrm{SO}(5)/\mathrm{SO}(3)$ -- which is now called the \emph{Berger space} -- where the embedding of $\mathrm{SO}(3)$ into $\mathrm{SO}(5)$ is given by its unique irreducible five-dimensional representation.
Finally, certain Aloff-Wallach spaces (cf.~{\cite{AloffWallachSpaces} for their definition}) also admit such positively curved Einstein metrics (cf.~{\cite[Theorem 3]{AloffWallachPositiveCurvatureMetrics}}).
All the aforementioned simply-connected spaces are \emph{homogeneous} manifolds; to our knowledge, there are no known examples of simply-connected, closed, non-homogeneous, Einstein manifolds with positive sectional curvature.

Before we discuss the proof of Theorem \ref{intro-th: sharp hessian estimate}, let us further contextualize the theorem.
As already mentioned, the first such inequality was Hamilton's matrix Li-Yau inequality for positive solutions to the heat equation (cf.~{\cite{HamiltonLiYau}}).
Other such matrix inequalities for flows have since been developed in other settings, such as for the heat equation on K\"{a}hler manifolds with non-negative holomorphic bisectional curvature (cf.~{\cite{CaoNiMatrixHeatKaehler}}), for the Ricci flow (cf.~{\cite{HamiltonHarnackRicciFlow}}) and the K\"{a}hler-Ricci flow (cf.~{\cite{NiMatrixKaehlerRicci}}), and for the mean curvature flow (cf.~{\cite{HamiltonMCF}}); we also refer the reader to Ni's survey from {\cite{NiSurvey}} for other such results and their applications.

In the elliptic setting, {\cite{ParkHessianEstimate}} shows, in particular, that for a complete, non-parabolic Riemannian manifold $(\widetilde{M}, \widetilde{g})$ of dimension $n \geq 3$ with non-negative sectional curvature and parallel Ricci curvature, if $\widetilde{G}$ is the minimal, positive Green's function for the Laplacian of $\widetilde{g}$ with singularity at a fixed point of $\widetilde{M}$, and if $\widetilde{s} := \widetilde{G}^{1/(2 - n)}$, then
\begin{equation} \label{intro-eq: hessian inequality for non-parabolic manifolds}
    \hess{\widetilde{s} \, {}^2} \leq C \tilde{g}
\end{equation}
away from the singularity of $\widetilde{G}$ for a constant $C \geq 10$, provided the same inequality holds at infinity on $\widetilde{M}$.
More recently, Park has also proved a similar Hessian inequality for the minimal, positive Green's function of the Laplacian of K\"{a}hler manifolds with non-negative holomorphic bisectional curvature (cf.~{\cite{ParkHessianEstimateKaehlerManifolds}}).
In this context, our result from Theorem \ref{intro-th: sharp hessian estimate} is a sharp, positively-curved analogue of \eqref{intro-eq: hessian inequality for non-parabolic manifolds}, which recovers the inequality from non-negative curvature in the formal limit as $k \downarrow 0$ (in the same manner as {\cite{ManeaSGE}} and {\cite{ManeaPSC}} did for their corresponding results), albeit losing sharpness and rigidity.

\subsection{The proof of the sharp Hessian inequality.} \label{intro-subsec: proof of sharp hessian inequality}
Let us now discuss the proof of Theorem \ref{intro-th: sharp hessian estimate}.
In this setting, if we denote by $\lambda$ the largest eigenvalue of $\mathcal{H} := \hess{s^2} + ks^2 g$, then Theorem \ref{intro-th: sharp hessian estimate} is equivalent to the statement that $\lambda \leq 2$ away from the singularity of $G$, with equality holding at \emph{some} point if and only if $(M, g)$ is isometric to the round $n$-sphere of constant curvature $k$; remark also that the asymptotics of $G$ near its singularity imply that $\lambda \to 2$ at $p$.
Our proof is similar to that of {\cite[Theorem 1.1]{ParkHessianEstimate}}: we will show that the inequality $\Delta \lambda - 2n \langle \nabla \log s, \nabla \lambda \rangle > 0$ holds in the barrier sense at every \emph{maximum} point of $\lambda$ where $\lambda$ is strictly greater than two.

At such a point $q$, if $m \geq 1$ denotes the multiplicity of $\lambda(q)$, $\lambda_1 \geq  \lambda_2 \geq \cdots \geq \lambda_n$ are the eigenvalues of $\mathcal{H}|_q$ with $\lambda(q) = \lambda_1 = \cdots = \lambda_m$, and if $(e_1, \cdots, e_n)$ is an orthonormal basis for $T_q M$ of corresponding eigenvectors of $\mathcal{H}$, the barrier for $\lambda$ at $q$ is explicitly given by the function $h := \mathcal{H}(V, V)$, where $V$ is a unit vector field with $V|_q = e_1$ and 
\begin{equation} \label{intro-eq: choice of vector field for barrier}
    \nabla_{e_i} V = \sum_{j = m + 1}^n \frac{(\nabla_{e_i} \mathcal{H})(e_j, e_1)}{\lambda(q) -\lambda_j} e_j,
\end{equation}
and we will show that $\Delta h(q) - 2n \langle \nabla \log s|_q, \nabla h|_q \rangle > 0$.

This barrier function is inspired by previous work in the setting of extrinsic geometric flows and constant rank theorems for solutions to elliptic equations (cf.~ {\cite[Lemma 5]{FlowsByPowerOfGaussCurvature}} and {\cite[Lemma 3.1]{ConstantRankTheoremsViscosity}}), where similar viscosity inequalities were shown for the Laplacian of the smallest eigenvalue of a two-tensor.
Other places where such techniques were used are, for instance, {\cite{ConstantRankTheoremNonlinearEquations}} and {\cite{GuanConvexitySolutions}}.
Moreover, by {\cite[Theorem 18.4]{AndrewsGeometricFlows}}, $\Delta h$ is precisely equal to $\Delta \lambda$ if $\lambda$ would be simple and smooth.
This barrier is thus a natural choice in the context of maximum principles for the eigenvalues of symmetric two-tensors.

\subsection{Sharp monotonic quantities associated to the Hessian inequality.} \label{intro-subsec: sharp monotonic quantities for the hessian}
We now turn to presenting various monotonic quantities associated to the sharp Hessian inequality from Theorem \ref{intro-th: sharp hessian estimate}, which will be positively-curved extensions of the results from {\cite{ParkHessianMonotonicity}}.
Monotonicity formulae play an important role in geometry and analysis because of their wide range of geometric implications.
In the setting of monotonicity for Green's functions, they were first proved by Colding (cf.~{\cite{ColdingNewMonotonicityFormulas}}) and Colding-Minicozzi (cf.~{\cite{ColdingMinicozziMonotonicityFormulas}}) for non-parabolic manifolds with non-negative Ricci curvature, and were used to prove uniqueness of tangent cones at infinity for Einstein manifolds (cf.~{\cite{CMUniquenessOfTangentCones}}).
These were recently extended by the first author to the setting of positive Ricci curvature from \eqref{intro-eq: sharp gradient estimate} in {\cite{ManeaSGE}}.

Various other monotonicity formulae have been defined and utilized in other settings, such as in the realm of potential theory (cf.~{\cite{MonotonicityAndGeometricInequalitiesForHypersurfacesInNonNegativeRicci, MinkowskiIneqViaPotentialTheory, MonotonicityInPotentialTheory}}), where they were used to prove Minkowski inequalities on Euclidean space and Willmore inequalities for closed hypersurfaces in manifolds with non-negative Ricci curvature.
Analogous formulae were also studied on manifolds with non-negative scalar curvature by Munteanu and Wang in {\cite{MunteanuWangGeometryOfPSC, MunteanuWangComparisonPaper}} -- recently extended by the first author to the setting of positive scalar curvature in {\cite{ManeaPSC}} -- and in {\cite{PenroseIneqViaGreenFunction, PMTViaGreenFunction, AreaCapacityIneqViaPotentialTheory}}, where the authors used such monotonic quantities to prove the positive mass theorem, the Riemannian Penrose inequality, and other capacity inequalities for three-manifolds.

Our first result in this direction is a two-parameter family of sharp monotonic quantities related to the Hessian inequality and the flow of $\nabla s^2$, providing a sharp, positively-curved extension of {\cite[Theorem 1.4]{ParkHessianMonotonicity}}.

\begin{theorem} \label{intro-th: sharp family of monotonicity formulae}
Let $(M, g)$ be a closed, connected, Einstein manifold of dimension $n \geq 3$, with sectional curvature bounded below by $k > 0$.
Let $G$ be Green's function for the operator $-\Delta + n(n - 2)k/4$ with singularity at a fixed point $p \in M$, put $s := G^{\frac{1}{2 - n}}$, and denote by $\varphi^t$ the flow of $\nabla s^2$ at time $t \in \R$.

Then, for any proper, precompact, open subset $U \subset M$ with smooth boundary and for which $\overline{s(U)} \subset (0, 2/\sqrt{k})$, for all $\beta > 0$, and for all $\alpha \geq n - \beta$, we have{\footnote{The integrals in Theorem \ref{intro-th: sharp family of monotonicity formulae} are taken with respect to the $n$-dimensional Hausdorff measure.}}
\begin{equation} \label{intro-eq: sharp family of monotonicity formulae}
    \frac{\dd}{\dd t} \left( e^{-2 \beta t} \int_{\varphi^t(U)} s^{-\alpha} \left(1 - \frac{k}{4} s^2 \right)^{\frac{\alpha - 2n}{2}} \left( \|\nabla s\|^2 + \frac{k}{4} s^2 \right)^{\frac{\beta}{2}} \right) \leq 0,
\end{equation}
and equality holds at \emph{some} $t \in \R$ if and only if $\alpha = n - \beta$ and $(M, g)$ is isometric to the round $n$-sphere of curvature $k$.
\end{theorem}

For the proof of Theorem \ref{intro-th: sharp family of monotonicity formulae}, we will compute explicitly the left-hand side of \eqref{intro-eq: sharp family of monotonicity formulae} (cf.~\eqref{eq: sharp montonicity formulae part 1}), from which its monotonicity and rigidity will follow.
The motivation for this theorem is the following.
If $(M, g)$ is the round $n$-sphere of curvature $k$ and $\varphi^t$ is the flow of $\nabla s^2$ at time $t$, we will see in Section \ref{sec: sharp monotonocity formulae} (cf.~\eqref{eq: canonical radial measure is preserved by the flow on the sphere} and Lemma \ref{lemma: forms scaled by flow of nabla s squared on the sphere}) that the only non-zero radial $n$-forms $\omega =  f \vol$ -- where $f$ depends only on the distance to the singularity of $G$ and $\!\vol$ is the volume form of the sphere -- for which $(\varphi^t)^*\omega = h_\omega(t) \omega$ for some function $h_\omega$ on $\R$ are, up to a constant, precisely the $n$-forms $\omega_\beta = s^{\beta - n} (1 - ks^2/4)^{-(n + \beta)/2} \vol$ for $\beta \in \R$, for which $h_{\omega_\beta}(t) = e^{2 \beta t}$.
For a general manifold as in the hypothesis of Theorem \ref{intro-th: sharp family of monotonicity formulae} and for $\alpha = n - \beta$, weighting these forms by the gradient quantity $(\|\nabla s\|^2 + ks^2/4)^{\beta/2}$ gives us a one-parameter family of sharp monotonic quantities, with corresponding global rigidity statements.
As in {\cite{ParkHessianMonotonicity}}, this may then be extended to a two-parameter family of monotonicity formulae, which is the context of Theorem \ref{intro-th: sharp family of monotonicity formulae}.

We will also present another sharp monotonic quantity related to the results from {\cite[Section 4]{ManeaSGE}}; let us briefly recall that context.
Under the weaker hypothesis that $\ric \geq (n - 1)kg$, the first author defined in {\cite[Section 4]{ManeaSGE}} two (locally absolutely continuous) functions $A$ and $V$ given by{\footnote{In this paper, for a continuous function on an $n$-manifold, the integrals over its level sets are taken with respect to the $(n - 1)$-dimensional Hausdorff measure; those over its sublevel sets are with respect to the $n$-dimensional Hausdorff measure.
We will suppress these measures from the notation.}} 
\begin{equation} \label{intro-eq: definition of A function from SGE paper}
    A(r) := r^{1 - n} \int_{s = r} \left( \|\nabla s\|^2 + \frac{k}{4} s^2 \right) \|\nabla s\| + \frac{nk}{4} \int_{s \leq r} \left( \|\nabla s\|^2 + \frac{k}{4} s^2 \right) G
\end{equation}
and
\begin{equation} \label{intro-eq: definition of V function from SGE paper}
    V(r) := r^{-n} \int_{s \leq r} \left( \|\nabla s\|^2 + \frac{k}{4} s^2 \right) \left( \|\nabla s\|^2 - \frac{k}{4} s^2 \right) + \frac{k}{4} \int_{s \leq r} \left( \|\nabla s\|^2 + \frac{k}{4} s^2 \right) G,
\end{equation}
which act as Green's function analogues of the scale-invariant volume of spheres and balls, with $A \leq \omega_{n - 1}$ and $V \leq \omega_{n - 1}/n$.
It was then proved that $A$ is non-increasing (cf.~{\cite[Theorem 1.2]{ManeaSGE}}), that $A - 2(n - 1)V$ is non-decreasing (cf.~{\cite[Theorem 4.11]{ManeaSGE}}), and that
\begin{equation} \label{intro-eq: monotonicity of r to 2 - n times A minus volume of unit sphere}
    \frac{\dd}{\dd r} \left( r^{2 - n} (A(r) - \omega_{n - 1}) \right) = \frac{r^{1 - n}}{2} \int_{s \leq r} s^{-n} \left( \|\!\tracelessHess{s^2}\|^2 + \ric(\nabla s^2, \nabla s^2) - (n - 1)k \|\nabla s^2\|^2 \right),
\end{equation}
for almost every $r$ in the range of $s$; the right-hand side of \eqref{intro-eq: monotonicity of r to 2 - n times A minus volume of unit sphere} is non-negative, and equality to zero forces rigidity; taking the limit as $k \downarrow 0$ in \eqref{intro-eq: definition of A function from SGE paper}--\eqref{intro-eq: monotonicity of r to 2 - n times A minus volume of unit sphere} recovers the monotonicity formulae from {\cite{ColdingNewMonotonicityFormulas}}.

Since $V$ is non-increasing and $V \leq \omega_{n - 1}/n$, the product rule immediately implies that the quantity $r^2(V(r) - \omega_{n - 1}/n)$ is non-increasing.
However, we may explicitly compute the derivative of this function to also obtain a global rigidity statement, analogous to that from {\cite[Theorem 1.6]{ParkHessianMonotonicity}}.
In this spirit, it becomes a companion sharp monotonic quantity for \eqref{intro-eq: monotonicity of r to 2 - n times A minus volume of unit sphere}.

\begin{theorem} \label{intro-th: monotonicity of r squared times V minus volume of unit ball}
Let $(M, g)$ be a closed, connected, Einstein manifold of dimension $n \geq 3$, with sectional curvature bounded below by $k > 0$, and let $G$ be Green's function for the operator $-\Delta + n(n - 2)k/4$, with singularity at a fixed point $p \in M$; put $s := G^{\frac{1}{2 - n}}$.

Then, for almost every $r$ in the range of $s$, we have
\begin{equation} \label{intro-eq: monotonicity of r squared times V minus volume of unit ball}
    \begin{split}
        \frac{\dd}{\dd r} \left( r^2 \left( V(r) - \frac{\omega_{n - 1}}{n} \right) \right) &= r^{1 - n} \int_{s \leq r} (\hess{s}^2 + ks^2g - 2g)(\nabla s, \nabla s) \\
        &\quad \, + \frac{k}{2} r^{1 - n} \int_{s \leq r}  s^2 \left( 1 - \frac{r^n}{s^n} \right) \left( 1 - \|\nabla s\|^2 - \frac{k}{4} s^2 \right).
    \end{split}
\end{equation}
Moreover, the right-hand side of \eqref{intro-eq: monotonicity of r squared times V minus volume of unit ball} vanishes at some $r > 0$ if and only if $(M, g)$ is isometric to the round $n$-sphere of constant curvature equal to $k$.
\end{theorem}

\subsection{Applications of the Hessian inequality.}
Besides the sharp monotonic quantities from Paragraph \ref{intro-subsec: sharp monotonic quantities for the hessian}, we will also discuss two other applications of the Hessian inequality from Theorem \ref{intro-th: sharp hessian estimate}.
Our first such result is a sharp comparison inequality (with respect to the round sphere) for the sum of two distinct Green's functions, whose singularities will act as ``comparison antipodal points''.
In this spirit, recently (cf.~{\cite{ParkRigidityPositiveCurvature}}), a Weyl tensor rigidity for the round four-sphere has been shown using a similar idea involving the quotient of two such Green's functions.
Our theorem takes the following form.

\begin{theorem} \label{intro-th: pole-antipodal point pointwise inequality for greens functions}
Let $(M, g)$ be a closed, connected, Einstein manifold of dimension $n \geq 3$, with sectional curvature bounded below by $k > 0$, and denote by $L := - \Delta + n(n - 2)k/4$.

Then, there exist distinct points $p, q \in M$ so that, if $G_p$ and $G_q$ denote Green's functions for $L$ with singularities at $p$ and $q$, respectively, and if $s_p := G_p^{1/(2 - n)}$ and $s_q := G_q^{1/(2 - n)}$, then
\begin{equation} \label{intro-eq: pole-antipodal point pointwise inequality for greens functions}
    s_p^2 + s_q^2 \leq \frac{4}{k},
\end{equation}
and equality holds in \eqref{intro-eq: pole-antipodal point pointwise inequality for greens functions} at \emph{some} point of $M \setminus \{p, q\}$ if and only if $(M, g)$ is isometric to the round $n$-sphere of constant curvature equal to $k$.
\end{theorem}

On the round $n$-sphere of curvature $k$, one may directly check that \eqref{intro-eq: pole-antipodal point pointwise inequality for greens functions} holds as an equality whenever $p$ and $q$ are antipodal points.
In general, since we always have the inequality $s_p \leq 2 \snk(\dist(p, \cdot)/2)$ (by {\cite[Proposition 2.2]{ManeaSGE}}), the diameter comparison theorem tells us that $s_p^2 \leq 4/k$, and thus \eqref{intro-eq: pole-antipodal point pointwise inequality for greens functions} is an improved comparison inequality for two distinct Green's functions.

For the proof of Theorem \ref{intro-th: pole-antipodal point pointwise inequality for greens functions}, we will actually show a more general comparison estimate (cf.~Theorem \ref{th: pole-antipodal point inequality}) which will imply this theorem.
Specifically, we will prove that, if $G_p$ is Green's function for $L := - \Delta + n(n - 2)k/4$ with singularity at $p$ and if $s_p := G_p^{1/(2 - n)}$, then, if $m > 0$ is the maximal value of $s_p$ on $M$ and if $q \in M \setminus \{p\}$ is so that $s_p(q) = m$, we have
\begin{equation} \label{intro-eq: pole-antipodal point inequality}
    s_p^2 + \left( \frac{k m^2}{2} - 1 \right) s_q^2 \leq m^2,
\end{equation}
where $s_q := G_q^{1/(2 - n)}$ and $G_q$ is Green's function for $L$ with singularity at $q$.
A similar rigidity statement holds for \eqref{intro-eq: pole-antipodal point inequality} as well: equality holds at some \emph{point} in $M \setminus \{p, q\}$ if and only if $(M, g)$ is isometric to the round $n$-sphere of constant curvature equal to $k$.

In the context of \eqref{intro-eq: pole-antipodal point inequality}, the round $n$-sphere has the property that $s_p$ has a unique maximum point, which is the antipodal point of $p$.
In general, it is not necessarily the case that if $q$ is a maximum point of $s_p$, then $p$ is a maximum point of $s_q$; however, two such points always exist by compactness of $M$.
The inequality \eqref{intro-eq: pole-antipodal point inequality} accounts for this.
Theorem \ref{intro-th: pole-antipodal point pointwise inequality for greens functions} follows by applying \eqref{intro-eq: pole-antipodal point inequality} to a pair of points $p$ and $q$ for which $p$ is also a maximum point of $s_q$.
The proof of \eqref{intro-eq: pole-antipodal point inequality} will employ a similar idea as in the proof of Obata's rigidity theorem for the sphere (cf.~{\cite{ObataSphericalRigidity}}): the Hessian estimate from Theorem \ref{intro-th: sharp hessian estimate} will allow us to show that the function $u := m^2 - s_p^2 - (-1 + km^2/2) s_q^2$ satisfies the inequality $\hess{u} + kug \geq 0$; a differential equation comparison argument along geodesics starting at $q$ then necessarily shows that $u$ must be non-negative, with a corresponding rigidity counterpart.

Our last application will be the inequality
\begin{equation} \label{intro-eq: hessian warped product deficit controlled by gradient deficit}
    4n \left( \|\nabla s\|^2 + \frac{k}{4} s^2 - 1 \right)^2 \leq \|\!\hess{s^2} + ks^2g - 2g\|^2 \leq 4n^2 \left( \|\nabla s\|^2 + \frac{k}{4} s^2 - 1 \right)^2,
\end{equation}
which follows immediately from the Hessian inequality from Theorem \ref{intro-th: sharp hessian estimate} and the linear algebra inequality that, for a symmetric, negative semi-definite matrix $A$, we have that $\|A\| \leq - \tr(A)$.
By the rigidity theory of warped products from {\cite{CheegerColdingWarpedProducts}}, in our setting of an Einstein $n$-manifold $(M, g)$ with positive sectional curvature, we know that if $\hess{s^2} + ks^2 g - 2g$ is close to zero (in an averaged $L^2$ sense), then $(M, g)$ is Gromov-Hausdorff close to the round $n$-sphere of curvature $k$.
In this spirit, \eqref{intro-eq: hessian warped product deficit controlled by gradient deficit} shows that if the ``gradient deficit'' $\|\nabla s\|^2 + ks^2/4 - 1$ is close to $0$, then so is the ``Hessian deficit'' $\hess{s^2} + ks^2 g - 2g$, and thus $(M, g)$ must be Gromov-Hausdorff close to the round $n$-sphere of curvature $k$.
A similar idea has been employed in {\cite{ParkRigidityPositiveCurvature}}, where the same conclusion has been shown under the hypothesis that the averaged integral of $(1 - \|\nabla s\|^2 - ks^2/4) G$ is close to zero.

\subsection{The Hessian inequality for non-parabolic manifolds.}
We end this introduction with a short discussion about the sharp analogue of the Hessian inequality in the setting \eqref{intro-eq: hessian inequality for non-parabolic manifolds}.
Sticking to that notation, in the spirit of Theorem \ref{intro-th: sharp hessian estimate}, the corresponding sharp estimate (by comparison to Euclidean space) would be $\hess{\widetilde{s} \, {}^2} \leq 2\widetilde{g}$, with equality forcing rigidity (with respect to Euclidean space); however, this is not the case.
We will show in Appendix \ref{appendix-sec: hessian inequalilty on noncompact manifolds} that, whilst $\hess{\widetilde{s} \, {}^2} = 2\tilde{g} + o(1)$ near the singularity of $\widetilde{G}$, it is possible that the corresponding inequality fails arbitrarily close to the pole (cf.~Theorem \ref{appendix-th: counter examples for non-compact case}).
In particular, we will see that this happens for manifolds of the form $\R^m \times N_1 \times \cdots \times N_\ell$ with $m$ sufficiently large, where $N_1, \cdots, N_\ell$ are closed, Einstein manifolds with positive scalar curvature and non-negative sectional curvature (cf.~Corollary \ref{appendix-cllr: failure of hessian inequality for products}).
However, we will also discuss that our proof of Theorem \ref{intro-th: sharp hessian estimate} extends to all non-parabolic, flat manifolds.

\subsection{Organization of the paper.}
Section \ref{sec: sharp hessian inequality} contains the proof of Theorem \ref{intro-th: sharp hessian estimate}, whilst in Section \ref{sec: sharp monotonocity formulae}, we prove Theorem \ref{intro-th: sharp family of monotonicity formulae} and \eqref{intro-eq: monotonicity of r squared times V minus volume of unit ball}.
Section \ref{sec: applications} is dedicated to showing Theorem \ref{intro-th: pole-antipodal point pointwise inequality for greens functions} and \eqref{intro-eq: hessian warped product deficit controlled by gradient deficit}.
Appendix \ref{appendix-sec: asymptotics of Greens function near the pole on closed manifolds} contains the proof of the $C^2$ asymptotics of Green's functions on closed manifolds near their singularity.
Finally, in Appendix \ref{appendix-sec: hessian inequalilty on noncompact manifolds}, we discuss the Hessian inequality for non-parabolic manifolds with non-negative curvature.

\subsection*{Acknowledgements.} We would like to thank our advisors, Tobias Colding and William Minicozzi, for their numerous helpful discussions and continuous support.
Cosmin Manea would also like to thank Jiewon Park for a valuable conversation about {\cite{ParkHessianEstimate}}.
During this project, Cosmin Manea was partially supported by NSF DMS Grant 2405393, and Jacob Reznikov was supported by the Natural Sciences and Engineering Research Council of Canada (NSERC) [grant number CGRSD-609347-2026].
\section{The sharp Hessian inequality} \label{sec: sharp hessian inequality}

In this section, we prove Theorem \ref{intro-th: sharp hessian estimate}.
Let $(M, g)$ be a closed, connected, Einstein manifold of dimension $n \geq 3$, with sectional curvature bounded below by $k > 0$.
Fix a point $p \in M$ and let $G$ be Green's function for the operator $-\Delta + n(n - 2)k/4$, with singularity at $p$.
By the strong maximum principle, $G$ is positive on $M \setminus \{p\}$; put $s := G^{\frac{1}{2 - n}}$.

Before we begin the proof of Theorem \ref{intro-eq: sharp hessian estimate}, let us first set the curvature conventions and notations we will use in the rest of this manuscript.
We will employ the following convention for the curvature tensor: we will denote by $R$ the $(1, 3)$-version of the Riemann curvature tensor, defined as
\begin{equation}
    R(X, Y)Z := \nabla_X \nabla_Y Z - \nabla_Y \nabla_X Z - \nabla_{[X, Y]} Z
\end{equation}
for vector fields $X, Y$, and $Z$ on $M$, and by $Rm$ the $(0, 4)$-version of $R$.
With this notation, for two linearly independent vectors $u$ and $v$, the sectional curvature of the plane spanned by $u$ and $v$ is defined by the equation
\begin{equation} \label{sec: sectional curvature convention}
    \sec(u, v) \equiv \frac{Rm(u, v, v, u)}{\|u\|^2 \|v\|^2 - \langle u, v \rangle^2}.
\end{equation}
We will also denote by $\Delta$ the (non-positive) Laplacian acting on both functions and tensors on $M$; we recall that if $(E_1, \cdots, E_n)$ is a local orthonormal frame for $M$, then $\Delta$ is defined as 
\begin{equation} \label{eq: Laplacian convention}
    \Delta \equiv \sum_{i = 1}^n \left( \nabla_{E_i} \nabla_{E_i} - \nabla_{\nabla_{E_i} E_i} \right).
\end{equation}
Moreover, for covariant two-tensors $T$ and $S$, we will denote by $T \circ S$ the $(0, 2)$-version of the composition of the endomorphisms of $TM$ induced by $T$ and $S$.

Let us now proceed by discussing the proof idea for Theorem \ref{intro-th: sharp hessian estimate}; set
\begin{equation} \label{eq: definition of H tensor}
    \mathcal{H} := \hess{s^2} + ks^2 g.
\end{equation}
Denote by $\lambda : M \setminus \{p\} \to \R$ the largest eigenvalue of $\mathcal{H}$ (viewed as a symmetric endomorphism of $TM$ through $g$); remark that this is a continuous function on $M \setminus \{p\}$.
To prove Theorem \ref{intro-th: sharp hessian estimate}, we need to show that $\lambda \leq 2$ on $M \setminus \{p\}$, with equality holding at \emph{some} point if and only if $(M, g)$ is isometric to the round $n$-sphere of curvature $k$.
By the asymptotics of Green's function near its singularity (cf.~Theorem \ref{appendix-th: hessian asymptotics of greens function in positive curvature} and Corollary \ref{appendix-cllr: sharp c2 asymptotics for s near singularity}), we have that $\mathcal{H} = 2g + o(1)$ as $x \to p$, and so we may extend $\lambda$ to a continuous function on $M$ by setting $\lambda(p) := 2$.
As mentioned in Section \ref{sec: intro} (cf.~Paragraph \ref{intro-subsec: proof of sharp hessian inequality}), we will show that $(\Delta - 2n \nabla_{\nabla \log s}) \lambda > 0$ in the barrier sense at all local maxima of $\lambda$ where $\lambda > 2$.
This would then immediately imply the desired result; however, for the sake of clarity, we will present the proof directly in the form of a contradiction argument.

Suppose that $q \in M \setminus \{p\}$ is a local maximum point for $\lambda$ with $\lambda(q) > 2$.
Denote by $m \geq 1$ the multiplicity of $\lambda(q)$ and by $\lambda_1 = \lambda_2 = \cdots = \lambda_m > \lambda_{m + 1} \geq \cdots \geq \lambda_n$ the eigenvalues of $\mathcal{H}|_q$, with $\lambda_1 = \cdots = \lambda_m = \lambda(q)$.
Since the chain rule immediately tells us that
\begin{equation} \label{eq: Delta s squared}
    \Delta s^2 = 2n \|\nabla s\|^2 - \frac{nk}{2} s^2,
\end{equation}
it follows that $\tr(\mathcal{H}) = \Delta s^2 + nk s^2 = 2n (\|\nabla s\|^2 + ks^2/4)$.
The sharp gradient estimate (cf.~\eqref{intro-eq: sharp gradient estimate} and {\cite[Theorem 1.1]{ManeaSGE}}) then implies that $\tr(\mathcal{H}) \leq 2n$, with equality holding at \emph{some} point if and only if $(M, g)$ is isometric to the round $n$-sphere of curvature $k$, in which case $\mathcal{H} \equiv 2g$.

Since $\lambda(q) > 2$, it follows that the multiplicity $m$ of $\lambda(q)$ satisfies the inequality $1 \leq m < n$.
Let $(e_1, \cdots, e_n)$ be an orthonormal basis for $T_q M$ of eigenvectors of $\mathcal{H}$, with $\mathcal{H}(e_i, e_j) = \delta_{ij} \lambda_i$.
As mentioned in Paragraph \ref{intro-subsec: proof of sharp hessian inequality} in Section \ref{sec: intro}, we define the barrier function $h$ for $\lambda$ as
\begin{equation} \label{eq: barrier function}
    h := \mathcal{H}(V, V),
\end{equation}
where $V$ is a unit vector field defined near $q$ with $V|_q = e_1$ and
\begin{equation} \label{eq: derivative of the vector field for barrier function}
    \nabla_{e_i} V = \sum_{j = m + 1}^n \frac{(\nabla_{e_i} \mathcal{H})(e_j, e_1)}{\lambda(q) - \lambda_j} e_j
\end{equation}
for all $i \in \{1, \cdots, n\}$.
Formally, $V$ is defined as follows.
Let $(E_1, \cdots, E_n)$ be a local orthonormal frame for $M$ near $q$ with $E_i|_q = e_i$ and $\nabla E_i|_q = 0$, and let $(x^1, \cdots, x^n)$ be normal coordinates centred at $q$ with $\partial_i|_q = e_i$; we may then set
\begin{equation} \label{eq: formal definition of vector field for barrier function}
    V := \|Y\|^{-1} Y, \text{ where } Y := E_1 + \sum_{j = m + 1}^n \sum_{i = 1}^n \frac{(\nabla_{e_i} \mathcal{H})(e_j, e_1)}{\lambda(q) - \lambda_j} x^i E_j,
\end{equation}
and a straightforward computation shows that $V|_q = e_1$ and that $V$ satisfies \eqref{eq: derivative of the vector field for barrier function}.

By construction and the variational characterization of eigenvalues of $\mathcal{H}$, $h$ is a smooth function with $h \leq \lambda$ and $h(q) = \lambda(q)$, and thus $h$ is a lower support for $\lambda$ at $q$; moreover, $q$ is a local maximum point for $h$, as it is one for $\lambda$.
The rest of this section is devoted to proving that $\Delta h(q) - 2n \langle \nabla \log s|_q, \nabla h|_q \rangle > 0$, thus giving us a contradiction to the existence of such a point $q$.
The motivation for this choice of barrier comes from {\cite[Lemma 5]{FlowsByPowerOfGaussCurvature} and {\cite[Lemma 3.1]{ConstantRankTheoremsViscosity}}}; a straightforward reformulation of those statements (rephrasing them for the largest eigenvalue of a tensor instead of the smallest one) allows us to restate those lemmas in the following form: if $\varphi$ is a smooth function that is an upper support for $\lambda$ at $q$ (\ie $\varphi \geq \lambda$ and $\varphi(q) = \lambda(q)$), then 
\begin{equation} \label{eq: support viscosity inequality for the barrier}
    \Delta \varphi(q) \geq (\Delta \mathcal{H})(e_1, e_1) + 2 \sum_{i = 1}^n \sum_{j = m + 1}^n \frac{\big( (\nabla_{e_i} \mathcal{H})(e_j, e_1) \big)^2}{\lambda(q) - \lambda_j}.
\end{equation}
The proof of this identity in {\cite[Lemma 5]{FlowsByPowerOfGaussCurvature}} is actually done through a barrier function argument analogously defined as in \eqref{eq: barrier function}.
Equation \eqref{eq: support viscosity inequality for the barrier} shows that $\Delta \lambda$ satisfies an inequality in the viscosity sense at $q$, and the right-hand side from \eqref{eq: support viscosity inequality for the barrier} would be equal to $\Delta \lambda(q)$ if $\lambda$ would be simple and twice differentiable at $q$.
In this sense, as our proof of Theorem \ref{intro-th: sharp hessian estimate} revolves around proving that $\Delta \lambda$ satisfies an inequality in a weak sense, these previous results and \eqref{eq: support viscosity inequality for the barrier} would formally allow us to prove a pointwise inequality for the Laplacian of $\lambda$ (at least at the points where $\lambda$ is smooth).
We refer the reader to {\cite[Theorem 18.4]{AndrewsGeometricFlows}} for other general computations of the derivatives of such eigenvalues.

In the spirit of the above, we will show that, at the point $q$,
\begin{equation} \label{eq: master identity for the laplacian of the barrier function}
    \Delta h - \frac{2n}{s} \langle \nabla s, \nabla h \rangle = (\Delta \mathcal{H})(e_1, e_1) - \frac{2n}{s} (\nabla_{\nabla s} \mathcal{H})(e_1, e_1) + 2 \sum_{i = 1}^n \sum_{j = m + 1}^n \frac{\big( (\nabla_{e_i} \mathcal{H})(e_j, e_1) \big)^2}{\lambda - \lambda_j},
\end{equation}
and we will show that the above right-hand side is positive.
However, it is worth remarking that the term $(\Delta \mathcal{H})(e_1, e_1) - 2n (\nabla_{\nabla \log s} \mathcal{H})(e_1, e_1)$ is not, a priori, non-negative; it is the entire right-hand side of \eqref{eq: master identity for the laplacian of the barrier function} that is a positive quantity, and we will see that this is a consequence of $g$ being an Einstein metric (cf.~equations \eqref{eq: sharp hessian inequality proof part 2} and \eqref{eq: final computation for the divergence term in laplace of barrier function}).
Explicitly, we will show that the last term on the right-hand side of \eqref{eq: master identity for the laplacian of the barrier function} may be bounded below by an expression involving $(\div \mathcal{H}(e_1))^2$, which is related to the Ricci tensor of $g$ as $\mathcal{H} \equiv \hess{s^2} + k s^2 g$.
For this reason, our proof does not extend to greater generality (to a setting of non-Einstein manifolds), and, in the limit $k \downarrow 0$, it does not specialize to a proof in the setting of non-parabolic manifolds with parallel Ricci tensor and non-negative sectional curvature of Park (cf.~{\cite{ParkHessianEstimate}}).
However, our proof does extend to all flat manifolds; we will discuss this more in Appendix \ref{appendix-sec: hessian inequalilty on noncompact manifolds}.

Let us now proceed with the proof of \eqref{eq: master identity for the laplacian of the barrier function}.
We record this as the lemma below.

\begin{lemma} \label{lemma: laplacian of barrier test function}
In the setting of \eqref{eq: barrier function}, at the point $q$, we have
\begin{equation} \label{eq: laplacian of barrier test function}
    \Delta h - \frac{2n}{s} \nabla_{\nabla s} h = (\Delta \mathcal{H})(e_1, e_1) - \frac{2n}{s} (\nabla_{\nabla s} \mathcal{H}) (e_1, e_1) + 2 \sum_{i = 1}^n \sum_{j = m + 1}^n \frac{\big( (\nabla_{e_i} \mathcal{H})(e_j, e_1) \big)^2}{\lambda - \lambda_j}.
\end{equation}
\end{lemma}

\begin{proof}
Let $(E_1, \cdots, E_n)$ be a local orthonormal frame for $M$ near $q$, with $E_i|_q = e_i$ and $\nabla E_i|_q = 0$ for all $i \in \{1, \cdots, n\}$.
By the product rule, for every vector field $X$ on $M$, we have
\[
\nabla_X h = (\nabla_X \mathcal{H})(V, V) + 2 \mathcal{H}(\nabla_X V, V)
\]
and differentiating this equation, we also obtain that, at the point $q$,
\[
\Delta h \equiv \sum_{i = 1}^n \nabla_{e_i} \nabla_{e_i} h = (\Delta \mathcal{H})(e_1, e_1) + 2 \mathcal{H} (\Delta V, e_1) + 4 \sum_{i = 1}^n (\nabla_{e_i} \mathcal{H})(\nabla_{e_i} V, e_1) + 2 \sum_{i = 1}^n \mathcal{H}(\nabla_{e_i} V, \nabla_{e_i} V).
\]
It follows that, at $q$,
\allowdisplaybreaks
\begin{align*}
    \Delta h - \frac{2n}{s} \nabla_{\nabla s} h &= (\Delta \mathcal{H})(e_1, e_1) + 2 \mathcal{H} (\Delta V, e_1) + 4 \sum_{i = 1}^n (\nabla_{e_i} \mathcal{H})(\nabla_{e_i} V, e_1) + 2 \sum_{i = 1}^n \mathcal{H}(\nabla_{e_i} V, \nabla_{e_i} V) \\
    &\quad \, - \frac{2n}{s} (\nabla_{\nabla s} \mathcal{H})(e_1, e_1) - \frac{4n}{s} \mathcal{H}(\nabla_{\nabla s} V, e_1).
\end{align*}
By construction of the vector field $V$ (cf.~\eqref{eq: derivative of the vector field for barrier function}), we have that $\nabla_w V \perp \mathrm{span}(e_1, \cdots, e_m)$ for all $w \in T_q M$, and so $\mathcal{H}(\nabla_{\nabla s} V, e_1) = 0$.
We thus have that, at the point $q$,
\allowdisplaybreaks
\begin{equation} \label{eq: laplace of barrier function before tensor laplacian computation part 1}
    \begin{split}
        \Delta h - \frac{2n}{s} \nabla_{\nabla s} h &= (\Delta \mathcal{H})(e_1, e_1) - \frac{2n}{s} (\nabla_{\nabla s} \mathcal{H})(e_1, e_1) \\
        &\quad \, + 2 \mathcal{H} (\Delta V, e_1) + 4 \sum_{i = 1}^n (\nabla_{e_i} \mathcal{H})(\nabla_{e_i} V, e_1) + 2 \sum_{i = 1}^n \mathcal{H}(\nabla_{e_i} V, \nabla_{e_i} V).
    \end{split}
\end{equation}

Let us now explicitly compute the last three terms on the right-hand side of \eqref{eq: laplace of barrier function before tensor laplacian computation part 1}.
Using \eqref{eq: derivative of the vector field for barrier function}, we immediately see that
\begin{equation} \label{eq: laplace of barrier function before tensor laplacian computation part 2}
    4 \sum_{i = 1}^n (\nabla_{e_i} \mathcal{H})(\nabla_{e_i} V, e_1) = 4 \sum_{i = 1}^n \sum_{j = m + 1}^n \frac{\big( (\nabla_{e_i} \mathcal{H})(e_j, e_1) \big)^2}{\lambda - \lambda_j}
\end{equation}
and, since $\mathcal{H}(e_j, e_k) = \delta_{jk} \lambda_k$, that
\begin{equation} \label{eq: laplace of barrier function before tensor laplacian computation part 3}
    2 \sum_{i = 1}^n \mathcal{H}(\nabla_{e_i} V, \nabla_{e_i} V) = 2 \sum_{i = 1}^n \sum_{j = m + 1}^n \lambda_j \left( \frac{(\nabla_{e_i} \mathcal{H})(e_j, e_1)}{\lambda - \lambda_j} \right)^2.
\end{equation}
To compute $2 \mathcal{H} (\Delta V, e_1)$, note that, by the product rule, we have $0 = \Delta \|V\|^2 = 2\langle \Delta V, V \rangle + 2 \|\nabla V\|^2$, and thus, by \eqref{eq: derivative of the vector field for barrier function}, at the point $q$,
\[
2\langle \Delta V, e_1 \rangle = - 2 \|\nabla V|_q\|^2 = - 2 \sum_{i = 1}^n \sum_{j = m + 1}^n \left( \frac{(\nabla_{e_i} \mathcal{H})(e_j, e_1)}{\lambda - \lambda_j} \right)^2,
\]
so
\begin{equation} \label{eq: laplace of barrier function before tensor laplacian computation part 4}
    2\mathcal{H}(\Delta V, e_1) = -2 \sum_{i = 1}^n \sum_{j = m + 1}^n \lambda \left( \frac{(\nabla_{e_i} \mathcal{H})(e_j, e_1)}{\lambda - \lambda_j} \right)^2.
\end{equation}
Combining \eqref{eq: laplace of barrier function before tensor laplacian computation part 2}, \eqref{eq: laplace of barrier function before tensor laplacian computation part 3}, and \eqref{eq: laplace of barrier function before tensor laplacian computation part 4}, we conclude that
\[
2 \mathcal{H} (\Delta V, e_1) + 4 \sum_{i = 1}^n (\nabla_{e_i} \mathcal{H})(\nabla_{e_i} V, e_1) + 2 \sum_{i = 1}^n \mathcal{H}(\nabla_{e_i} V, \nabla_{e_i} V) = 2 \sum_{i = 1}^n \sum_{j = m + 1}^n \frac{\big( (\nabla_{e_i} \mathcal{H})(e_j, e_1) \big)^2}{\lambda - \lambda_j}.
\]
The proof is now finished by inserting this preceding identity into \eqref{eq: laplace of barrier function before tensor laplacian computation part 1}.
\end{proof}

The next step in the proof of Theorem \ref{intro-th: sharp hessian estimate} is to compute the expression for $\Delta \mathcal{H} \equiv \Delta\!\hess{s^2} + k (\Delta s^2) g$.
For this, we recall that $\Delta\!\hess{s^2}$ is related to $\hess{\Delta s^2}$ via the action of the Lichnerowicz Laplacian on two-tensors (cf.~{\cite[Equation (1.180b)]{BesseRG}} for its definition); since $(M, g)$ is Einstein (and in particular, $\nabla\!\ric = 0$), this gives us the explicit formula
\begin{equation} \label{eq: commuting Hessian with Laplacian}
    \Delta\!\hess{s^2} = \hess{\Delta s^2} - 2 Rm(\hess{s^2}) + \ric \circ \hess{s^2} + \hess{s^2} \circ \ric,
\end{equation}
where $Rm(\hess{s^2})$ is defined via the action of the Riemann tensor on symmetric, covariant two-tensors as $Rm(\hess{s^2})(X, Y) := \langle \hess{s^2}, Rm(\cdot, X, Y, \cdot) \rangle$ for vector fields $X$ and $Y$ on $M$; a proof of this identity may be found in {\cite[Lemma 2.33]{ChowRicciFlow}} or in {\cite[Lemma 3.1]{ParkHessianEstimate}}.
Combining \eqref{eq: commuting Hessian with Laplacian} with \eqref{eq: Delta s squared}, it follows that
\begin{equation} \label{eq: Delta Hess s squared}
    \Delta\!\hess{s^2} = 2n \hess{\|\nabla s\|^2} - \frac{nk}{2} \hess{s^2} - 2 Rm(\hess{s^2}) + \ric \circ \hess{s^2} + \hess{s^2} \circ \ric.
\end{equation}

We now want to rewrite the right-hand side of \eqref{eq: Delta Hess s squared} in terms of $\mathcal{H}$ and the Riemann curvature tensor of $(M, g)$.
We start with the term $\hess{\|\nabla s\|^2}$; for this, we employ the ``Bochner formula'' for the Hessian of $\|\nabla s\|^2$ and the chain rule.

\begin{lemma} \label{lemma: hessian of gradient of s squared in terms of H}
On $M \setminus \{p\}$, we have
\begin{equation} \label{eq: hessian of gradient of s squared in terms of H}
    \begin{split}
        \hess{\|\nabla s\|^2} &= \frac{1}{s} \nabla_{\nabla s} \mathcal{H} + \frac{1}{2s^2} \mathcal{H}^2 - \frac{1}{s^2} \left( \|\nabla s\|^2 + k s^2 \right) \mathcal{H} + \frac{8}{s^2} \left( \|\nabla s\|^2 + \frac{k}{2} s^2 \right) ds \otimes ds \\
        &\quad \, - \frac{2}{s^2} \left( \mathcal{H} \circ (ds \otimes ds) + (ds \otimes ds) \circ \mathcal{H} \right) - k \left( \|\nabla s\|^2 - \frac{k}{2} s^2 \right) g + 2 Rm(\cdot, \nabla s, \nabla s, \cdot),
    \end{split}
\end{equation}
where $\mathcal{H}^2 := \mathcal{H} \circ \mathcal{H}$.
\end{lemma}

\begin{proof}
By the ``Bochner formula'' (the Riccati equation) for the Hessian of $\|\nabla s\|^2$ (cf.~{\cite[Theorem 3.2.2]{PetersenRG}}), we have
\begin{equation} \label{eq: bochner formula at hessian level for s}
    \frac{1}{2} \hess{\|\nabla s\|^2} = \nabla_{\nabla s} \hess{s} + \left( \hess{s} \right)^2 + Rm(\cdot, \nabla s, \nabla s, \cdot),
\end{equation}
where $(\hess{s})^2 := \hess{s} \circ \hess{s}$.
We now rewrite the first two terms on the right-hand side of \eqref{eq: bochner formula at hessian level for s} in terms of $\mathcal{H}$.
For the first one, since
\begin{equation} \label{eq: hessian of s in terms of hessian of s squared}
    \hess{s} = \frac{1}{2s} \hess{s^2} - \frac{1}{s} ds \otimes ds = \frac{1}{2s} \mathcal{H} - \frac{k}{2} s g - \frac{1}{s} ds \otimes ds,
\end{equation}
it follows that
\allowdisplaybreaks
\begin{align} \label{align: nabla nabla s hess s in terms of Hess s squared part 1}
    \nabla_{\nabla s} \hess{s} &= \nabla_{\nabla s} \left( \frac{1}{2s} \mathcal{H} - \frac{k}{2} s g - \frac{1}{s} ds \otimes ds \right) \nonumber \\
    &= \frac{1}{2s} \nabla_{\nabla s} \mathcal{H} - \frac{\|\nabla s\|^2}{2s^2} \mathcal{H} - \frac{k}{2} \|\nabla s\|^2 g + \frac{\|\nabla s\|^2}{s^2} ds \otimes ds - \frac{1}{s} \left( \nabla_{\nabla s} ds \otimes ds + ds \otimes \nabla_{\nabla s} ds \right) \nonumber \\
    &= \frac{1}{2s} \nabla_{\nabla s} \mathcal{H} - \frac{\|\nabla s\|^2}{2s^2} \mathcal{H} - \frac{k}{2} \|\nabla s\|^2 g + \frac{\|\nabla s\|^2}{s^2} ds \otimes ds \nonumber \\
    &\quad \, - \frac{1}{s} \left( \hess{s}(\nabla s, \cdot) \otimes ds + ds \otimes \hess{s}(\nabla s, \cdot) \right),
\end{align}
where the last equality follows because $\nabla ds \equiv \hess{s}$, so $\nabla_{\nabla s} ds = \hess{s}(\nabla s, \cdot)$.
Now, using \eqref{eq: hessian of s in terms of hessian of s squared} again, we may further compute
\allowdisplaybreaks
\begin{align*}
    \hess{s}(\nabla s, \cdot) \otimes ds &= \left( \frac{1}{2s} \mathcal{H}(\nabla s, \cdot) - \frac{k}{2} s \, ds - \frac{\|\nabla s\|^2}{s} ds \right) \otimes ds \\
    &= \frac{1}{2s} \mathcal{H}(\nabla s, \cdot) \otimes ds - \frac{k}{2} s \, ds \otimes ds - \frac{\|\nabla s\|^2}{s} ds \otimes ds,
\end{align*}
and one may directly check that $\mathcal{H}(\nabla s, \cdot) \otimes ds = (ds \otimes ds) \circ \mathcal{H}$, hence
\begin{equation} \label{eq: hess s otimes ds in terms of s squared}
    \hess{s}(\nabla s, \cdot) \otimes ds = \frac{1}{2s} (ds \otimes ds) \circ \mathcal{H} - \frac{k}{2} s \, ds \otimes ds - \frac{\|\nabla s\|^2}{s} ds \otimes ds.
\end{equation}
A similar computation holds for $ds \otimes \hess{s}(\nabla s, \cdot)$; combining \eqref{eq: hess s otimes ds in terms of s squared} with \eqref{align: nabla nabla s hess s in terms of Hess s squared part 1} and grouping the common terms, we obtain that
\begin{equation} \label{eq: nabla nabla s hess s in terms of hess s squared}
    \begin{split}
        \nabla_{\nabla s} \hess{s} &= \frac{1}{2s} \nabla_{\nabla s} \mathcal{H} - \frac{\|\nabla s\|^2}{2s^2} \mathcal{H} - \frac{k}{2} \|\nabla s\|^2 g + k \, ds \otimes ds + \frac{3 \|\nabla s\|^2}{s^2} ds \otimes ds \\
        &\quad \, - \frac{1}{2s^2} \left( \mathcal{H} \circ (ds \otimes ds) + (ds \otimes ds) \circ \mathcal{H} \right).
    \end{split}
\end{equation}

The second term from the right-hand side of \eqref{eq: bochner formula at hessian level for s} may be rewritten using \eqref{eq: hessian of s in terms of hessian of s squared} as
\allowdisplaybreaks
\begin{align} \label{align: square of hess s in terms of hess s squared}
    (\hess{s})^2 &= \left( \frac{1}{2s} \mathcal{H} - \frac{k}{2} s g - \frac{1}{s} ds \otimes ds \right)^2 \nonumber \\
    &= \frac{1}{4s^2} \mathcal{H}^2 + \frac{k^2}{4} s^2 g + \frac{\|\nabla s\|^2}{s^2} ds \otimes ds - \frac{k}{2} \mathcal{H} \nonumber \\
    &\quad \, + k \, ds \otimes ds - \frac{1}{2s^2} \left( \mathcal{H} \circ (ds \otimes ds) + (ds \otimes ds) \circ \mathcal{H} \right),
\end{align}
where the last equality follows because $(ds \otimes ds)^2 = \|\nabla s\|^2 ds \otimes ds$.
The proof of the desired identity is now completed by inserting \eqref{eq: nabla nabla s hess s in terms of hess s squared} and \eqref{align: square of hess s in terms of hess s squared} into \eqref{eq: bochner formula at hessian level for s}.
\end{proof}

Combining \eqref{eq: hessian of gradient of s squared in terms of H} from Lemma \ref{lemma: hessian of gradient of s squared in terms of H} with \eqref{eq: commuting Hessian with Laplacian} and rewriting the remaining terms using $\mathcal{H}$, we obtain the second-order differential equation that $\mathcal{H}$ satisfies.

\begin{corollary} \label{cllr: differential equation satisfied by H}
On $M \setminus \{p\}$, we have
\begin{equation} \label{eq: second order equation satisfied by H}
    \begin{split}
        \Delta \mathcal{H} - \frac{2n}{s} \nabla_{\nabla s} \mathcal{H} &= 4n Rm(\cdot, \nabla s, \nabla s, \cdot) - 2 Rm(\mathcal{H}) + \ric \circ \mathcal{H} + \mathcal{H} \circ \ric \\
        &\quad \, + \frac{n}{s^2} \mathcal{H}^2 - \frac{1}{s^2} \left( 2n \|\nabla s\|^2 + \frac{5nk}{2} s^2 \right) \mathcal{H} \\
        &\quad \, - \frac{4n}{s^2} \left( \mathcal{H} \circ (ds \otimes ds) + (ds \otimes ds) \circ \mathcal{H} \right) \\
        &\quad \, + \frac{16n}{s^2} \left( \|\nabla s\|^2 + \frac{k}{2} s^2 \right) ds \otimes ds + n k^2 s^2 g,
    \end{split}
\end{equation}
where $\mathcal{H}^2 := \mathcal{H} \circ \mathcal{H}$.
\end{corollary}

\begin{proof}
By \eqref{eq: Delta s squared}, we have
\begin{equation} \label{eq: Delta H part 1}
    \Delta \mathcal{H} \equiv \Delta( \hess{s^2} + ks^2 g) = \Delta \! \hess{s^2} + 2nk \|\nabla s\|^2 g - \frac{nk^2}{2} s^2 g.
\end{equation}
Now, we rewrite every term on the right-hand side of \eqref{eq: Delta Hess s squared} in terms of $\mathcal{H}$ and the curvature tensor of $(M, g)$.
The first term is handled by Lemma \ref{lemma: hessian of gradient of s squared in terms of H}.
For the others, we have
\begin{equation} \label{eq: Delta H part 2}
    -2 Rm(\hess{s^2}) = - 2 Rm(\mathcal{H} - ks^2 g) = - 2 Rm(H) + 2k s^2 \ric,
\end{equation}
where the last equality follows because $Rm(g) \equiv \ric$,
\begin{equation} \label{eq: Delta H part 3}
    \ric \circ \hess{s^2} = \ric \circ (\mathcal{H} - ks^2 g) = \ric \circ \mathcal{H} - ks^2 \ric,
\end{equation}
and similarly, $\hess{s^2} \circ \ric = \mathcal{H} \circ \ric - k s^2 \ric$.
Combining \eqref{eq: Delta H part 2} with \eqref{eq: Delta H part 3}, it follows that
\[
- 2 Rm(\hess{s^2}) + \ric \circ \hess{s^2} + \hess{s^2} \circ \ric = - 2 Rm(\mathcal{H}) + \ric \circ \mathcal{H} + \mathcal{H} \circ \ric.
\]
Inserting this preceding equation and \eqref{eq: hessian of gradient of s squared in terms of H} from Lemma \ref{lemma: hessian of gradient of s squared in terms of H} into \eqref{eq: Delta Hess s squared}, we obtain that
\begin{equation} \label{eq: Delta H part 4}
    \begin{split}
        \Delta \! \hess{s^2} &= 4n Rm(\cdot, \nabla s, \nabla s, \cdot) - 2 Rm(\mathcal{H}) + \ric \circ \mathcal{H} + \mathcal{H} \circ \ric \\
        &\quad \, + \frac{2n}{s} \nabla_{\nabla s} \mathcal{H} + \frac{n}{s^2} \mathcal{H}^2 - \frac{1}{s^2} \left( 2n \|\nabla s\|^2 + \frac{5nk}{2} s^2 \right) \mathcal{H} \\
        &\quad \, - \frac{4n}{s^2} \left( \mathcal{H} \circ (ds \otimes ds) + (ds \otimes ds) \circ \mathcal{H} \right) \\
        &\quad \, + \frac{16n}{s^2} \left( \|\nabla s\|^2 + \frac{k}{2} s^2 \right) ds \otimes ds - 2 nk \|\nabla s\|^2 g + \frac{3nk^2}{2} s^2 g.
    \end{split}
\end{equation}
The proof of this corollary is now completed by combining \eqref{eq: Delta H part 4} with \eqref{eq: Delta H part 1}.
\end{proof}

Let us now further rewrite the first two curvature terms on the right-hand side of \eqref{eq: second order equation satisfied by H} from Corollary \ref{cllr: differential equation satisfied by H} so as to better show how the sectional curvature lower bound emerges in the proof of Theorem \ref{intro-th: sharp hessian estimate}.
Recall that, if the sectional curvature of $(M, g)$ is constant, equal to $k$, then $Rm = k g \kn g$, where $\kn$ denotes the Kulkarni-Nomizu product; this would imply that $Rm(\cdot, \nabla s, \nabla s, \cdot) = k(\|\nabla s\|^2 g - ds \otimes ds)$ and that $Rm(T) = k \tr(T) g - kT$ for all symmetric, covariant two-tensors $T$.
In general, since we only assume that $\sec \geq k$, let us define
\begin{equation} \label{eq: excess curvature tensor in the direction of nabla s}
    \widehat{Rm}_{\nabla s} := Rm(\cdot, \nabla s, \nabla s, \cdot) - k \|\nabla s\|^2 g + k \, ds \otimes ds
\end{equation}
and
\begin{equation} \label{eq: excess curvature action on two tensors}
    \widehat{Rm}(T) := Rm(T) - k \tr(T) g + k T,
\end{equation}
for a symmetric, covariant two-tensor $T$ on $M$; these quantities measure the excess sectional curvature in the direction of $\nabla s$ and with respect to the action of the Riemann tensor on symmetric, covariant two-tensors.
Since $\tr(\mathcal{H}) = \Delta s^2 + nks^2 = 2n \|\nabla s\|^2 + nks^2/2$, using \eqref{eq: excess curvature tensor in the direction of nabla s} and \eqref{eq: excess curvature action on two tensors}, we may rewrite
\allowdisplaybreaks
\begin{align} \label{align: excess curvature terms in Delta H}
    4n Rm(\cdot, \nabla s, \nabla s, \cdot) - 2 Rm(\mathcal{H}) &= 4n \widehat{Rm}_{\nabla s} + 4nk \|\nabla s\|^2 g - 4nk \, ds \otimes ds \nonumber \\
    &\quad \, - 2 \widehat{Rm}(\mathcal{H}) - 4nk \|\nabla s\|^2 g - nk^2 s^2 g + 2k \mathcal{H} \nonumber \\
    &= 4n \widehat{Rm}_{\nabla s} - 2 \widehat{Rm}(\mathcal{H}) + 2k \mathcal{H} - nk^2 s^2 g - 4 nk \, ds \otimes ds.
\end{align}

We may then obtain the following final rewriting of the elliptic equation satisfied by $\mathcal{H}$.

\begin{lemma} \label{lemma: final form of Delta H}
On $M \setminus \{p\}$, we have
\begin{equation} \label{eq: final form of Delta H}
    \begin{split}
        \Delta \mathcal{H} - \frac{2n}{s} \nabla_{\nabla s} \mathcal{H} &= 4n \widehat{Rm}_{\nabla s} - 2 \widehat{Rm}(\mathcal{H}) + \ric \circ \mathcal{H} + \mathcal{H} \circ \ric - 2(n - 1)k \mathcal{H} \\
        &\quad \, + \frac{n}{s^2} \left(\mathcal{H} - 4 \, ds \otimes ds \right)^2 - \frac{2n}{s^2} \left( \|\nabla s\|^2 + \frac{k}{4} s^2 \right) \mathcal{H} + 4 n k \, ds \otimes ds.
    \end{split}
\end{equation}
\end{lemma}

\begin{proof}
Combining \eqref{align: excess curvature terms in Delta H} with Corollary \ref{cllr: differential equation satisfied by H}, it follows that
\allowdisplaybreaks
\begin{align*}
    \Delta \mathcal{H} - \frac{2n}{s} \nabla_{\nabla s} \mathcal{H} &= 4n \widehat{Rm}_{\nabla s} - 2 \widehat{Rm}(\mathcal{H}) + \ric \circ \mathcal{H} + \mathcal{H} \circ \ric \\
    &\quad \, + 2k \mathcal{H} - n k^2 s^2 g - 4 nk ds \otimes ds \\
    &\quad \, + \frac{n}{s^2} \mathcal{H}^2 - \frac{1}{s^2} \left( 2n \|\nabla s\|^2 + \frac{5nk}{2} s^2 \right) \mathcal{H} \\
    &\quad \, - \frac{4n}{s^2} \left( \mathcal{H} \circ (ds \otimes ds) + (ds \otimes ds) \circ \mathcal{H} \right) \\
    &\quad \, + \frac{16n}{s^2} \left( \|\nabla s\|^2 + \frac{k}{2} s^2 \right) ds \otimes ds + n k^2 s^2 g,
\end{align*}
and the desired identity now follows by grouping common terms and recalling that $(ds \otimes ds)^2 = \|\nabla s\|^2 ds \otimes ds$.
\end{proof}

Let us now use Lemma \ref{lemma: final form of Delta H} to compute the first two terms on the right-hand side of \eqref{eq: laplacian of barrier test function} from Lemma \ref{lemma: laplacian of barrier test function}.
We record this computation in the form of the following corollary.

\begin{corollary} \label{cllr: Delta H applied to an eigenvector of H}
In the setting and notation from \eqref{eq: barrier function}, we have, at the point $q$,
\begin{equation} \label{eq: Delta H applied to an eigenvector of H}
    \begin{split}
        (\Delta \mathcal{H}) &(e_1, e_1) - \frac{2n}{s} \left(\nabla_{\nabla s} \mathcal{H} \right) (e_1, e_1) \\
        &= 4n \widehat{Rm}_{\nabla s}(e_1, e_1) - 2 \widehat{Rm}(\mathcal{H})(e_1, e_1) + (\ric \circ \mathcal{H})(e_1, e_1) + (\mathcal{H} \circ \ric)(e_1, e_1) \\
        &\quad \, - 2(n - 1)k \mathcal{H}(e_1, e_1) + \frac{n}{s^2} \left( \lambda - 8 \langle \nabla s, e_1 \rangle^2 \right) \left( \lambda - 2 \left( \|\nabla s\|^2 + \frac{k}{4} s^2 \right) \right).
    \end{split}
\end{equation}
\end{corollary}

\begin{proof}
Since $(ds \otimes ds)^2 = \|\nabla s\|^2 ds \otimes ds$ and $\mathcal{H}(e_1, e_1) = \lambda(q)$, one immediately computes that
\allowdisplaybreaks
\begin{align*}
    \frac{n}{s^2} &\left(\mathcal{H} - 4 \, ds \otimes ds \right)^2(e_1, e_1) - \frac{2n}{s^2} \left( \|\nabla s\|^2 + \frac{k}{4} s^2 \right) \mathcal{H}(e_1, e_1) + 4 n k \, (ds \otimes ds)(e_1, e_1) \\
    &= \frac{n}{s^2} \left( \lambda - 8 \langle \nabla s, e_1 \rangle^2 \right) \left( \lambda - 2 \left( \|\nabla s\|^2 + \frac{k}{4} s^2 \right) \right).
\end{align*}
The desired result now follows by combining this preceding identity with \eqref{eq: final form of Delta H} from Lemma \ref{lemma: final form of Delta H}.
\end{proof}

Corollary \ref{cllr: Delta H applied to an eigenvector of H} allows us to rewrite the first two terms on the right-hand side of \eqref{eq: laplacian of barrier test function} from Lemma \ref{lemma: laplacian of barrier test function} in terms of $\lambda(q)$ and the curvature tensor of $(M, g)$, and we are almost finished with the proof of Theorem \ref{intro-th: sharp hessian estimate}.
We previously mentioned that, in general, $(\Delta \mathcal{H})(e_1, e_1) - 2n (\nabla_{\nabla \log s} \mathcal{H})(e_1, e_1)$ may not be non-negative when evaluated at an eigenvector of $\mathcal{H}$ (corresponding to its largest eigenvalue), and this is clearly seen to be a consequence of the last term on the right-hand side of  \eqref{eq: Delta H applied to an eigenvector of H}.

However, the other terms in that equation always combine to a non-negative quantity.
Indeed, since we assumed that the sectional curvature of $M$ is bounded below by $k$, by construction (cf.~\eqref{eq: excess curvature tensor in the direction of nabla s}),
\begin{equation} \label{eq: excess sectional curvature in direction of nabla s is non-negative}
    \widehat{Rm}_{\nabla s}(e_1, e_1) \geq 0.
\end{equation}

The non-negativity of the remaining terms is implied by the following lemma.

\begin{lemma} \label{lemma: non-negativity of the curvature terms on largest eigenvectors}
In the setting and notation from \eqref{eq: barrier function}, we have
\[
- 2 \widehat{Rm}(\mathcal{H})(e_1, e_1) + (\ric \circ \mathcal{H})(e_1, e_1) + (\mathcal{H} \circ \ric)(e_1, e_1) - 2(n - 1)k \mathcal{H}(e_1, e_1) \geq 0.
\]
\end{lemma}

\begin{proof}
Recall that, in the situation of \eqref{eq: barrier function}, we defined $\lambda_1, \cdots, \lambda_n$ to be the eigenvalues of $\mathcal{H}|_q$ with $\lambda_1 = \lambda(q)$, and $(e_1, \cdots, e_n)$ to be an orthonormal basis for $T_q M$ consisting of eigenvectors of $\mathcal{H}$, with $\mathcal{H}(e_i, e_j) = \delta_{ij} \lambda_i$.
Denoting by $R_{ijkl} := Rm(e_i, e_j, e_k, e_l)$, it then follows that $\sec(e_1, e_j) = R_{1jj1}$, that $Rm(\mathcal{H})(e_1, e_1) = \sum_{j = 2}^n \sec(e_1, e_j) \lambda_j$, and that
\begin{align} \label{align: non-negativity of the curvature terms on largest eigenvectors part 1}
    \widehat{Rm}(\mathcal{H})(e_1, e_1) &\equiv Rm(\mathcal{H})(e_1, e_1) - k \tr(\mathcal{H}|_q) + k \mathcal{H}(e_1, e_1) \nonumber \\
    &= \sum_{j = 2}^n \sec(e_1, e_j) \lambda_j - k \sum_{j = 1}^n \lambda_j + k \lambda_1 \nonumber \\
    &= \sum_{j = 2}^n (\sec(e_1, e_j) - k) \lambda_j.
\end{align}
Similarly, since $g$ is Einstein, $(\mathcal{H} \circ \ric)(e_1, e_1) = (\ric \circ \mathcal{H})(e_1, e_1) = \lambda_1 \ric(e_1, e_1)$, and as $\ric(e_1, e_1) = \sum_{j = 2}^n \sec(e_1, e_j)$, we get that
\begin{equation} \label{align: non-negativity of the curvature terms on largest eigenvectors part 2}
    (\mathcal{H} \circ \ric) (e_1, e_1) + (\ric \circ \mathcal{H}) (e_1, e_1) = 2 \sum_{j = 2}^n \sec(e_1, e_j) \lambda_1.
\end{equation}

Combining \eqref{align: non-negativity of the curvature terms on largest eigenvectors part 1} with \eqref{align: non-negativity of the curvature terms on largest eigenvectors part 2}, we obtain that
\allowdisplaybreaks
\begin{align*}
    -2 &\widehat{Rm}(\mathcal{H})(e_1, e_1) + (\mathcal{H} \circ \ric)(e_1, e_1) + (\ric \circ \mathcal{H})(e_1, e_1) - 2(n - 1)k \mathcal{H}(e_1, e_1) \\
    &= -2 \sum_{j = 2}^n (\sec(e_1, e_j) - k) \lambda_j + 2 \sum_{j = 2}^n \sec(e_1, e_j) \lambda_1 - 2(n - 1)k \lambda_1 \\
    &= 2 \sum_{j = 2}^n (\sec(e_1, e_j) - k)(\lambda_1 - \lambda_j),
\end{align*}
and this is non-negative as $\sec \geq k$ and $\lambda_1 \equiv \lambda(q) \geq \lambda_j$ for all $j \geq 2$.
\end{proof}

The final ingredient required to complete the proof of Theorem \ref{intro-th: sharp hessian estimate} is the following simple lemma, which shows that the covariant derivative of $\mathcal{H}$ evaluated at $\lambda(q)$-eigenvectors is zero.
This result is required to rewrite the last term on the right-hand side of \eqref{eq: master identity for the laplacian of the barrier function} using the divergence of $\mathcal{H}$.

\begin{lemma} \label{lemma: covariant derivative of H evaluated at largest eigenvectors is zero}
In the setting and notation of \eqref{eq: barrier function}, for all $\lambda(q)$-eigenvectors $u, v \in T_q M$ and all $w \in T_q M$, we have
\[
(\nabla_w \mathcal{H})(u, v) = 0.
\]
\end{lemma}

\begin{proof}
Let $u$ and $v$ be $\lambda(q)$-eigenvectors of $\mathcal{H}|_q$, and let $w \in T_q M$.
By linearity and the polarization identity, it suffices to show the desired conclusion only when $u = v$ and $\|u\| = 1$.
Let $U$ be a unit vector field defined near $q$ with $U|_q = u$ and $\nabla U|_q = 0$; define the function $f := \mathcal{H}(U, U)$.
As in the argument for the function $h$ from \eqref{eq: barrier function} (using the variational characterization of the eigenvalues of $\mathcal{H}|_q$), $f$ is smooth and $q$ is a local maximum point for $f$, hence $\nabla f|_q = 0$.
However, the product rule tells us that
\[
0 = \langle \nabla f|_q, w \rangle = (\nabla_w \mathcal{H})(u, u) + 2 \mathcal{H}(\nabla_w U, u) = (\nabla_w \mathcal{H})(u, u),
\]
where the last equality follows because, by construction, $\nabla_w U = 0$.
\end{proof}

We are now ready to prove Theorem \ref{intro-th: sharp hessian estimate}.

\begin{proof}[Proof of Theorem \ref{intro-th: sharp hessian estimate}]
As in the previous parts of this section, denote by $\lambda : M \setminus \{p\} \to \R$ the largest eigenvalue of $\mathcal{H} \equiv \hess{s^2} + k s^2 g$; this is a continuous function on $M \setminus \{p\}$ which can be extended by continuity to all of $M$ by setting $\lambda(p) := 2$ (by the asymptotics of $\mathcal{H}$ near $p$, cf.~Theorem \ref{appendix-th: hessian asymptotics of greens function in positive curvature} and Corollary \ref{appendix-cllr: sharp c2 asymptotics for s near singularity}).
To prove Theorem \ref{intro-th: sharp hessian estimate}, we need to show that $\lambda \leq 2$ on $M \setminus \{p\}$, with equality holding at \emph{some} point if and only if $(M, g)$ is isometric to the round $n$-sphere of constant curvature equal to $k$.

We start by proving that $\lambda \leq 2$ on $M \setminus \{p\}$.
Suppose, for the sake of contradiction, that $\sup_{M \setminus \{p\}} \lambda > 2$.
Since $M$ is compact and $\lambda$ is continuous on $M$, this implies that there is a point $q \in M \setminus \{p\}$ so that $\lambda(q) > 2$ and $q$ is a maximum point for $\lambda$. Denote by $m \geq 1$ the multiplicity of $\lambda(q)$; we previously showed (using \eqref{eq: Delta s squared} and the sharp gradient estimate for $s$, cf.~\eqref{intro-eq: sharp gradient estimate} or {\cite[Theorem 1.1]{ManeaSGE}}) that, at the point $q$,
\begin{equation} \label{eq: trace of H}
    \tr(\mathcal{H}) = 2n \left( \|\nabla s\|^2 + \frac{k}{4} s^2 \right) \leq 2n.
\end{equation}
Thus, we necessarily have that $1 \leq m < n$.
Denote by $\lambda_1 = \cdots = \lambda_m > \lambda_{m + 1} \geq \cdots \geq \lambda_n$ the eigenvalues of $\mathcal{H}|_q$, with $\lambda_1 = \cdots = \lambda_m = \lambda(q)$, and let $(e_1, \cdots, e_n)$ be an orthonormal basis for $T_q M$ consisting of eigenvectors of $\mathcal{H}$, with $\mathcal{H}(e_i, e_j) = \delta_{ij} \lambda_i$.

Consider the barrier function $h := \mathcal{H}(V, V)$ from \eqref{eq: barrier function} where $V$ is a smooth, unit vector field defined near $q$ with $V|_q = e_1$ and for which $\nabla V$ satisfies the identity from \eqref{eq: derivative of the vector field for barrier function}.
As discussed shortly after \eqref{eq: barrier function}, the function $h$ is a lower support for $\lambda$ at $q$, and $q$ is a local maximum point for $h$ by the variational characterization of $\lambda$; in particular, $\nabla h|_q = 0$ and $\Delta h(q) \leq 0$.

By Lemma \ref{lemma: laplacian of barrier test function}, at the point $q$, we have
\[
\Delta h - \frac{2n}{s} \nabla_{\nabla s} h = (\Delta \mathcal{H})(e_1, e_1) - \frac{2n}{s} (\nabla_{\nabla s} \mathcal{H}|_q) (e_1, e_1) + 2 \sum_{i = 1}^n \sum_{j = m + 1}^n \frac{\big( (\nabla_{e_i} \mathcal{H})(e_j, e_1) \big)^2}{\lambda - \lambda_j},
\]
and by \eqref{eq: Delta H applied to an eigenvector of H} in Corollary \ref{cllr: Delta H applied to an eigenvector of H}, \eqref{eq: excess sectional curvature in direction of nabla s is non-negative}, and Lemma \ref{lemma: non-negativity of the curvature terms on largest eigenvectors}, it follows that, at the point $q$,
\begin{equation} \label{eq: sharp hessian inequality proof part 1}
    \Delta h - \frac{2n}{s} \nabla_{\nabla s} h \geq \frac{n}{s^2} \left( \lambda - 8 \langle \nabla s, e_1 \rangle^2 \right) \left( \lambda - 2 \left( \|\nabla s\|^2 + \frac{k}{4} s^2 \right) \right) + 2 \sum_{i = 1}^n \sum_{j = m + 1}^n \frac{\big( (\nabla_{e_i} \mathcal{H})(e_j, e_1) \big)^2}{\lambda - \lambda_j}.
\end{equation}

Let us now further rewrite the last term on the right-hand side of \eqref{eq: sharp hessian inequality proof part 1}.
First, by discarding the summands with $i \neq j$ and then using the Cauchy-Schwarz inequality, we have
\begin{equation} \label{eq: sharp hessian inequality proof part 2}
    2 \sum_{i = 1}^n \sum_{j = m + 1}^n \frac{\big( (\nabla_{e_i} \mathcal{H})(e_j, e_1) \big)^2}{\lambda - \lambda_j} \geq 2 \sum_{j = m + 1}^n \frac{\big( (\nabla_{e_j} \mathcal{H})(e_j, e_1) \big)^2}{\lambda - \lambda_j} \geq \frac{2 \left( \sum_{j = m + 1}^n (\nabla_{e_j} \mathcal{H})(e_j, e_1) \right)^2}{\sum_{j = m + 1}^n (\lambda - \lambda_j)}.
\end{equation}
The denominator on the right-hand side of \eqref{eq: sharp hessian inequality proof part 2} may be explicitly computed (at $q$) as
\begin{equation} \label{eq: denominator in last sum term for proof of sharp hessian inequality}
    \sum_{j = m + 1}^n (\lambda - \lambda_j) = (n - m) \lambda - \sum_{j = m + 1}^n \lambda_j = n \lambda - \tr(\mathcal{H}) = n \left( \lambda - 2 \left( \|\nabla s\|^2 + \frac{k}{4} s^2 \right) \right),
\end{equation}
where the last equality follows from \eqref{eq: trace of H}.

The numerator may also be computed explicitly in terms of $\div \mathcal{H}(e_1)$ as follows.
Recall that we assumed that $g$ is an Einstein metric; this means that $\ric = \rho g$, for some constant $\rho \in \R$ satisfying the inequality $\rho \geq (n - 1)k$, as the sectional curvature of $g$ is bounded below by $k$.
Lemma \ref{lemma: covariant derivative of H evaluated at largest eigenvectors is zero} then tells us that $(\nabla_{e_j} \mathcal{H})(e_j, e_1) = 0$ for all $j \in \{1, \cdots, m\}$, and so
\[
(\div \mathcal{H})(e_1) \equiv \sum_{j = 1}^n  (\nabla_{e_j} \mathcal{H})(e_j, e_1) = \sum_{j = m + 1}^n (\nabla_{e_j} \mathcal{H})(e_j, e_1).
\]
Hence, together with \eqref{eq: denominator in last sum term for proof of sharp hessian inequality}, it follows that 
\begin{equation} \label{eq: last term for proof of sharp hessian inequality in divergence form}
    \frac{2 \left( \sum_{j = m + 1}^n (\nabla_{e_j} \mathcal{H})(e_j, e_1) \right)^2}{\sum_{j = m + 1}^n (\lambda - \lambda_j)} = \frac{2 \big(\!\div \mathcal{H}(e_1) \big)^2}{n \left( \lambda - 2 \left( \|\nabla s\|^2 + \frac{k}{4} s^2 \right) \right)}.
\end{equation}
However, since $\mathcal{H} \equiv \hess{s^2} + ks^2g$, we may also compute
\begin{equation} \label{eq: divergence of H}
    \div \mathcal{H} = d \Delta s^2 + \ric(\nabla s^2, \cdot ) + k \, d s^2 = 2n \, d \|\nabla s\|^2 + nks \, ds + (\rho - (n - 1)k) ds^2,
\end{equation}
where the last equality follows because $\ric(\nabla s^2, \cdot) = \rho \, d s^2$ and because $\Delta s^2 = 2n \|\nabla s\|^2 - nks^2/2$ (cf.~\eqref{eq: Delta s squared}).

Using the product rule and the identity $d \|\nabla s\|^2 = 2 \hess{s}(\nabla s, \cdot)$, we may then further rewrite, at the point $q$,
\allowdisplaybreaks
\begin{align*}
    \langle \nabla \|\nabla s\|^2, e_1 \rangle &= 2 \hess{s}(\nabla s, e_1) \\
    &= 2 \left( \frac{1}{2s} \hess{s^2}(\nabla s, e_1) - \frac{1}{s} \|\nabla s\|^2 \langle \nabla s, e_1 \rangle \right) \\
    &= \frac{1}{s} (\mathcal{H}(\nabla s, e_1) - k s^2 \langle \nabla s, e_1 \rangle) - \frac{2}{s} \|\nabla s\|^2 \langle \nabla s, e_1 \rangle \\
    &= \frac{1}{s} \left( \lambda -  2 \|\nabla s\|^2 \right) \langle \nabla s, e_1 \rangle - ks \langle \nabla s, e_1 \rangle,
\end{align*}
where the last identity follows because $e_1$ is a $\lambda(q)$-eigenvector of $\mathcal{H}|_q$.
Combining this preceding identity with \eqref{eq: divergence of H}, we obtain that
\begin{align} \label{align: divergence of H in the direction of a largest eigenvector}
    \div \mathcal{H}(e_1) &= 2n \left( \frac{1}{s} \left( \lambda -  2 \|\nabla s\|^2 \right) \langle \nabla s, e_1 \rangle - ks \langle \nabla s, e_1 \rangle \right) + nks \langle \nabla s, e_1 \rangle + 2s (\rho - (n - 1)k) \langle \nabla s, e_1 \rangle \nonumber \\
    &= \frac{2n}{s} \left( \lambda - 2 \left( \|\nabla s\|^2 + \frac{k}{4} s^2 \right) \right) \langle \nabla s, e_1 \rangle + 2s (\rho - (n - 1)k) \langle \nabla s, e_1 \rangle \nonumber \\
    &= 2s \langle \nabla s, e_1 \rangle \left( \frac{n}{s^2} \left( \lambda - 2 \left( \|\nabla s\|^2 + \frac{k}{4} s^2 \right) \right) + (\rho - (n - 1)k) \right).
\end{align}

Inserting \eqref{align: divergence of H in the direction of a largest eigenvector} into \eqref{eq: last term for proof of sharp hessian inequality in divergence form} and expanding the square, we get that
\begin{equation} \label{eq: final computation for the divergence term in laplace of barrier function}
    \begin{split}
        \frac{2\left( \sum_{j = m + 1}^n (\nabla_{e_j} \mathcal{H})(e_j, e_1) \right)^2}{\sum_{j = m + 1}^n (\lambda - \lambda_j)} &= \frac{8n}{s^2} \langle \nabla s, e_1 \rangle^2 \left( \lambda - 2 \left( \|\nabla s\|^2 + \frac{k}{4} s^2 \right) \right) \\
        &\quad \, + 16 \langle \nabla s, e_1 \rangle^2 (\rho - (n - 1)k) + \frac{8 \langle \nabla s, e_1 \rangle^2 s^2(\rho - (n - 1)k)^2}{n \left( \lambda - 2 \left( \|\nabla s\|^2 + \frac{k}{4} s^2 \right) \right)}.
    \end{split}
\end{equation}
Combining \eqref{eq: final computation for the divergence term in laplace of barrier function} with \eqref{eq: sharp hessian inequality proof part 2} and \eqref{eq: sharp hessian inequality proof part 1} and grouping the common terms, it follows that, at the point $q$, we have
\begin{equation} \label{eq: end of proof of sharp hessian inequality}
    \begin{split}
        \Delta h - \frac{2n}{s} \nabla_{\nabla s} h &\geq \frac{n \lambda}{s^2} \left( \lambda - 2 \left( \|\nabla s\|^2 + \frac{k}{4} s^2 \right) \right) \\
        &\quad \, + 16 \langle \nabla s, e_1 \rangle^2 (\rho - (n - 1)k) + \frac{8 \langle \nabla s, e_1 \rangle^2 s^2(\rho - (n - 1)k)^2}{n \left( \lambda - 2 \left( \|\nabla s\|^2 + \frac{k}{4} s^2 \right) \right)}.
    \end{split}
\end{equation}
Since, by assumption, $\lambda(q) > 2 \geq 2 (\|\nabla s\|^2 + ks^2/4)$, the first term on the right-hand side of \eqref{eq: end of proof of sharp hessian inequality} is strictly positive; the other two terms are also clearly non-negative.
We conclude that $\Delta h(q) - 2n \langle \nabla \log s|_q, \nabla h|_q \rangle > 0$, which is a contradiction as $q$ is a local maximum point for $h$.
This finishes the proof of the Hessian inequality for $s^2$.

For the rigidity statement, suppose that $\lambda(q) = 2$ at some point $q \in M \setminus \{p\}$; since the previous paragraphs above show that $\lambda \leq 2$ on $M \setminus \{p\}$, it follows that $q$ is a global maximum point for $\lambda$.
Again, suppose, for the sake of contradiction, that $(M, g)$ is not isometric to the round $n$-sphere of constant curvature $k$.
By the rigidity statement of the sharp gradient estimate for $s$ (cf.~{\cite[Theorem 1.1]{ManeaSGE}}), it follows that $(2n)^{-1} \tr(\mathcal{H}) = \|\nabla s\|^2 + ks^2/4 < 1$ on $M \setminus \{p\}$.
We may now define the same barrier function $h$ from \eqref{eq: barrier function} -- for which $q$ would be a local maximum point for $h$, \ie $\Delta h(q) \leq 0$ -- and go through the same argument from \eqref{eq: sharp hessian inequality proof part 1}--\eqref{eq: end of proof of sharp hessian inequality} to obtain that $\Delta h(q) > 0$, which would be a contradiction.
Thus, $(M, g)$ must be isometric to the round $n$-sphere of constant curvature equal to $k$.
\end{proof}
\section{Sharp monotonic quantities related to the Hessian inequality} \label{sec: sharp monotonocity formulae}

This section is dedicated to proving the sharp monotonic quantities described in Paragraph \ref{intro-subsec: sharp monotonic quantities for the hessian} in Section \ref{sec: intro}.
We start with the proof of Theorem \ref{intro-th: sharp family of monotonicity formulae}.
As mentioned in Section \ref{sec: intro}, let us first describe in more detail the context and motivation behind this theorem by first analysing the round sphere.
Throughout the rest of this section, we denote by $\S^n_k$ the round $n$-sphere of constant curvature equal to $k$, and by $\S^n := \S^n_1$.
Fix a point $p \in \S^n_k$; as discussed in the introduction, the operator $-\Delta + n(n - 2)k/4$ is the conformal Laplacian of $\S^n_k$, and its Green's function $G$ (with singularity at $p$) may be explicitly computed (by \cite[Lemma 2.1]{ManeaSGE}) as $G = (2 \snk(r/2))^{2 - n}$, where $r := \dist(p, \cdot)$ and $\snk$ is the curvature-adapted sine function.
In this case, $s \equiv 2 \snk(r/2)$ and $r = 2 \asnk(s/2)$, where $\asnk$ is the curvature-adapted arcsine function (the inverse of $\snk$ on $[-\pi/(2\sqrt{k}), \pi/(2\sqrt{k})]$).
Remark also that $s^2$ is smooth on all of $\S^n_k$.

If $\varphi^t : \S^n_k \to \S^n_k$ denotes the flow of $\nabla s^2$ at time $t$, then a direct computation shows that $\frac{d}{dt} (r \circ \varphi^t) = 2 \snk(r \circ \varphi^t)$, from which it immediately follows that 
\begin{equation} \label{eq: radisu }
    r \circ \varphi^t = 2 \atnk \left( e^{2t} \tnk(r/2) \right),
\end{equation}
where $\tnk$ is the curvature-adapted tangent function (defined as $\tnk := \snk/\csk$, with $\csk = \snk'$), and $\atnk$ is its inverse.
A direct computation then shows, if $\dd\!\vol_{\S^n_k}$ denotes the volume form of $\S^n_k$, then for all $\beta, t \in \R$, we have
\begin{equation} \label{eq: canonical radial measure is preserved by the flow on the sphere}
    (\varphi^t)^* \left( s^{\beta - n} \left( 1 - \frac{k}{4} s^2 \right)^{-\frac{n + \beta}{2}} \dd\!\vol_{\S^n_k} \right) = e ^{2 \beta t} \left( s^{\beta - n} \left( 1 - \frac{k}{4} s^2 \right)^{-\frac{n + \beta}{2}} \dd\!\vol_{\S^n_k} \right),
\end{equation}
that is, $\varphi^t$ scales the top-form $s^{\beta - n} (1 - ks^2/4)^{-(n + \beta)/2} \dd \vol_{\S^n_k}$ by a factor $e^{2 \beta t}$.
In fact, we can prove that, up to an angular function, all $n$-forms on $\S^n_k \setminus \{p, -p\}$ scaled by $\varphi^t$ are of this form, where $-p$ is the antipodal point of $p$.

\begin{lemma} \label{lemma: forms scaled by flow of nabla s squared on the sphere}
If $\omega = f \dd\!\vol_{\S^n_k}$ is a non-zero $n$-form on $\S^n_k \setminus \{p, -p\}$ for which $(\varphi^t)^*\omega = h(t) \omega$ for some (continuous) function $h$ on $\R$, then $h(t) = e^{2 \beta t}$ for some $\beta \in \R$ and, on $\S^n_k \setminus \{p, -p\} \simeq (0, \pi/\sqrt{k}) \times_{\snk^2} \S^{n - 1}$, $f(r, \theta) = C(\theta) s^{\beta - n} (1 - ks^2/4)^{-(n + \beta)/2}$ for some function $C$ on $\S^{n - 1}$.
\end{lemma}

\begin{proof}
The group property of $\varphi^t$ necessarily implies that $h(t + s) = h(t) h(s)$ for all $t, s \in \R$.
Thus, as $\omega \neq 0$, we get that $h(t) = e^{2\beta t}$ for some constant $\beta \in \R$.
Differentiating the identity $(\varphi^t)^*\omega = h(t) \omega$ at $t = 0$ and using the definition of $\omega$ and that of the divergence of a vector field, we obtain that $f$ satisfies the equation $\div(f \nabla s^2) = 2\beta f$.
Since $s = 2 \snk(r/2)$, this can be explicitly rewritten as
\[
2 \snk(r) \partial_r f + 2n \csk(r) f = 2\beta f,
\]
which, using the warped product structure of $\S^n_k \setminus \{p, -p\} \simeq (0, \pi/\sqrt{k}) \times_{\snk^2} \S^{n - 1}$, may be explicitly solved to obtain that 
\begin{equation} \label{eq: functions which give measures that are scaled by flow of nabla s squared on the sphere}
    f(r, \theta) = C(\theta) (2 \snk(r/2))^{\beta - n} \csk(r/2)^{-(\beta + n)} = C(\theta) s^{\beta - n} \left(1 - \frac{k}{4} s^2 \right)^{- \frac{n + \beta}{2}},
\end{equation}
for all $(r, \theta) \in (0, \pi/\sqrt{k}) \times \S^{n - 1}$, and $C$ is a function on $\S^{n - 1}$.
\end{proof}

In particular, since $\|\nabla s\|^2 + ks^2/4 \equiv 1$ on $\S^n_k$ (cf.~the rigidity statement of \eqref{intro-eq: sharp gradient estimate} and {\cite[Theorem 1.1]{ManeaSGE}}), \eqref{intro-eq: sharp family of monotonicity formulae} from Theorem \ref{intro-th: sharp family of monotonicity formulae} is an equality when $\alpha = n - \beta$.
We also remark that the only function of the form from \eqref{eq: functions which give measures that are scaled by flow of nabla s squared on the sphere} which is smooth (or at least continuous) on all of $\S^n_k$ is necessarily the zero function if $\beta \geq 0$.

In general, if $(M, g)$ is a closed, connected, Einstein manifold of dimension $n \geq 3$, with sectional curvature bounded below by $k > 0$, if $G$ is Green's function for the operator  $-\Delta + n(n - 2)k/4$, with singularity at a fixed point $p \in M$, and if $s := G^{\frac{1}{2 - n}}$, Theorem \ref{intro-th: sharp family of monotonicity formulae} provides a two-parameter family of (sharp) monotonic quantities when weighting the measures from \eqref{eq: canonical radial measure is preserved by the flow on the sphere} and Lemma \ref{lemma: forms scaled by flow of nabla s squared on the sphere} by a positive power of $\|\nabla s\|^2 + ks^2/4$, which measures the deficit to $\S^n_k$.
Remark also the asymptotics of $s^2$ near its singularity (cf.~Corollary \ref{appendix-cllr: sharp c2 asymptotics for s near singularity}) imply that $s^2$ is a $C^2$ function near $p$ if we define $s(p) := 0$ (with $\nabla s^2|_p = 0$ and $\hess{s^2}|_p = 2g$), and thus the flow of $\nabla s^2$ exists for all time.

Let us now proceed with the proof of Theorem \ref{intro-th: sharp family of monotonicity formulae}; this will be a positive-curvature extension of the proof of {\cite[Theorem 1.4]{ParkHessianMonotonicity}}.

\begin{proof}[Proof of Theorem \ref{intro-th: sharp family of monotonicity formulae}]
Let $\beta > 0$ and $\alpha \geq n - \beta$, and let $U \subset M$ be a proper, precompact subset of $M$ with smooth boundary; denote by
\begin{equation} \label{eq: sharp montonicity formulae definition of function}
    f_{\alpha, \beta} := s^{-\alpha} \left(1 - \frac{k}{4} s^2 \right)^{\frac{\alpha - 2n}{2}} \left( \|\nabla s\|^2 + \frac{k}{4} s^2 \right)^{\frac{\beta}{2}}.
\end{equation}
Since $U$ is a domain with smooth boundary and $\varphi^t$ is the flow of $\nabla s^2$, we know that
\[
\frac{\dd}{\dd t} \int_{\varphi^t(U)} f_{\alpha, \beta} = \int_{\varphi^t(U)} \div(f_{\alpha, \beta} \nabla s^2),
\]
and so, by the product rule,
\begin{equation} \label{eq: sharp montonicity formulae part 1}
    \frac{\dd}{\dd t} \left( e^{-2 \beta t} \int_{\varphi^t(U)} f_{\alpha, \beta} \right) = e^{-2 \beta t} \int_{\varphi^t(U)} f_{\alpha, \beta} \left( \Delta s^2 + \langle \nabla \log f_{\alpha, \beta}, \nabla s^2 \rangle - 2 \beta \right).
\end{equation}
We now proceed by computing the right-hand side of \eqref{eq: sharp montonicity formulae part 1} explicitly in terms of $\hess{s^2}$ and $\|\nabla s\|^2$.

Since $d \|\nabla s\|^2 = 2 \hess{s}(\nabla s, \cdot)$, the product rule allows us to write
\allowdisplaybreaks
\begin{align*}
    d \left( \|\nabla s\|^2 + \frac{k}{4} s^2 \right) &= 2 \hess{s}(\nabla s, \cdot) + \frac{k}{2} s \, d s \\
    &= \frac{1}{s} \hess{s^2}(\nabla s, \cdot) - \frac{2\|\nabla s\|^2}{s} \, ds + \frac{k}{2} s \, ds \\
    &= \frac{1}{s} \left( \hess{s^2} + ks^2 g - 2g \right) (\nabla s, \cdot) + \frac{2}{s} \left(1 - \|\nabla s\|^2 - \frac{k}{4} s^2 \right) ds,
\end{align*}
and combining this identity with \eqref{eq: sharp montonicity formulae definition of function} and grouping common terms, we obtain that
\allowdisplaybreaks
\begin{align} \label{align: sharp montonicity formulae part 2}
    \langle \nabla \log f_{\alpha, \beta}, \nabla s^2 \rangle &= \left\langle - \frac{\alpha}{s} \nabla s - \frac{k(\alpha - 2n)}{4} \frac{s}{1 - ks^2/4} \nabla s + \frac{\beta}{2 \left( \|\nabla s\|^2 + \frac{k}{4} s^2 \right)} \nabla \left( \|\nabla s\|^2 + k s^2/4 \right), \nabla s^2 \right\rangle \nonumber \\
    &= - 2 (\alpha + \beta) \|\nabla s\|^2 - \frac{k(\alpha - 2n)}{2} \frac{s^2}{1 - ks^2/4} \|\nabla s\|^2 \nonumber \\
    &\quad \, + \frac{\beta}{\|\nabla s\|^2 + k s^2/4} \left( \hess{s^2} + ks^2 g - 2g \right) (\nabla s, \nabla s) + \frac{2 \beta}{\|\nabla s\|^2 + k s^2/4} \|\nabla s\|^2.
\end{align}
Using the assumption that $\alpha + \beta \geq n$, we may further rewrite \eqref{align: sharp montonicity formulae part 2} as
\allowdisplaybreaks
\begin{align*}
    \langle \nabla \log f_{\alpha, \beta}, \nabla s^2 \rangle &= - 2n \|\nabla s\|^2 + \frac{k(n + \beta)}{2} \frac{s^2}{1 - ks^2/4} \|\nabla s\|^2 + \frac{2 \beta}{\|\nabla s\|^2 + k s^2/4} \|\nabla s\|^2 \\
    &\quad \, + \frac{\beta}{\|\nabla s\|^2 + k s^2/4} \left( \hess{s^2} + ks^2 g - 2g \right) (\nabla s, \nabla s) \\
    &\quad \, - 2(\alpha + \beta - n) \|\nabla s\|^2 - \frac{k(\alpha + \beta - n)}{2} \frac{s^2}{1 - ks^2/4} \|\nabla s\|^2.
\end{align*}
We may now combine this preceding identity with the equation $\Delta s^2 = 2n \|\nabla s\|^2 - nk s^2/2$ (cf.~\eqref{eq: Delta s squared}) and also factor $\|\nabla s\|^2$ in the last two terms on its right-hand side to obtain that
\begin{equation} \label{eq: sharp montonicity formulae part 3}
    \begin{split}
        \Delta s^2 + \langle \nabla \log f_{\alpha, \beta}, \nabla s^2 \rangle - 2 \beta &= - \frac{nk}{2} s^2 + \frac{k(n + \beta)}{2} \frac{s^2}{1 - ks^2/4} \|\nabla s\|^2 + \frac{2 \beta}{\|\nabla s\|^2 + k s^2/4} \|\nabla s\|^2 - 2\beta \\
        &\quad \, + \frac{\beta}{\|\nabla s\|^2 + k s^2/4} \left( \hess{s^2} + ks^2 g - 2g \right)(\nabla s, \nabla s) \\
        &\quad \, - 2(\alpha + \beta - n) \frac{\|\nabla s\|^2}{1 - ks^2/4}.
    \end{split}
\end{equation}
Finally, using that $\|\nabla s\|^2 = (\|\nabla s\|^2 + ks^2/4 - 1) + (1 - ks^2/4)$ for the first four terms on the right-hand side of the above equation, we may finally rewrite \eqref{eq: sharp montonicity formulae part 3} as
\begin{equation} \label{eq: sharp montonicity formulae part 4}
    \begin{split}
        \Delta s^2 + \langle \nabla \log f_{\alpha, \beta}, \nabla s^2 \rangle - 2 \beta &=\frac{k s^2}{2} \left( \frac{\beta}{\|\nabla s\|^2 + k s^2/4} + \frac{n + \beta}{1 - ks^2/4} \right) \left( \|\nabla s\|^2 + \frac{k}{4} s^2 - 1 \right) \\
        &\quad \, + \frac{\beta}{\|\nabla s\|^2 + k s^2/4} \left( \hess{s^2} + ks^2 g - 2g \right)(\nabla s, \nabla s) \\
        &\quad \, - 2 (\alpha + \beta - n) \frac{\|\nabla s\|^2}{1 - ks^2/4}.
    \end{split}
\end{equation}
The right-hand side of \eqref{eq: sharp montonicity formulae part 4} is non-positive by the sharp gradient estimate for $s$ (cf.~\eqref{intro-eq: sharp gradient estimate}), by Theorem \ref{intro-th: sharp hessian estimate}, and because $\alpha + \beta - n \geq 0$ and $s^2 \leq 4/k$ (by the diameter comparison theorem, cf.~{\cite[Proposition 2.2]{ManeaSGE}}).
Combining this with \eqref{eq: sharp montonicity formulae part 1}, the proof of the monotonicity statement of Theorem \ref{intro-th: sharp family of monotonicity formulae} is completed.

For the rigidity statement, suppose that the derivative of the quantity from \eqref{intro-eq: sharp family of monotonicity formulae} in Theorem \ref{intro-th: sharp family of monotonicity formulae} is zero at some $t \in \R$ and for some proper, precompact subset $U \subset M$ with smooth boundary.
This implies that the right-hand side of \eqref{eq: sharp montonicity formulae part 4} must be equal to zero almost everywhere on $\varphi^t(U)$.
This necessarily tells us that $\alpha = n - \beta$, that $\|\nabla s\|^2 + ks^2/4 = 1$, and that $\hess{s}^2 (\nabla s, \nabla s) + ks^2 \|\nabla s\|^2 = 2\|\nabla s\|^2$ almost everywhere on $\varphi^t(U)$.
The rigidity part of either Theorem \ref{intro-th: sharp hessian estimate} or {\cite[Theorem 1.1]{ManeaSGE}} then implies that $(M, g)$ must also be isometric to the round $n$-sphere of constant curvature equal to $k$.
\end{proof}

We continue with the proof of Theorem \ref{intro-th: monotonicity of r squared times V minus volume of unit ball}; recall the definition of the function $V$ from \eqref{intro-eq: definition of V function from SGE paper}:
\begin{equation} \label{eq: definition of V function}
    V(r) := r^{-n} \int_{s \leq r} \left( \|\nabla s\|^2 + \frac{k}{4} s^2 \right) \left( \|\nabla s\|^2 - \frac{k}{4} s^2 \right) + \frac{k}{4} \int_{s \leq r} \left( \|\nabla s\|^2 + \frac{k}{4} s^2 \right) G,
\end{equation}
defined for $r$ in the range of $s$.
As mentioned in Section \ref{sec: intro} (cf.~Paragraph \ref{intro-subsec: sharp monotonic quantities for the hessian}), $V$ is closely related to the sharp gradient for $s$ (cf.~\eqref{intro-eq: sharp gradient estimate}), and was studied by the first author in {\cite[Section 4]{ManeaSGE}}, where it was proved that $V \leq \omega_{n - 1}/n$, that $V$ is non-increasing, and that $V$ is part of another sharp monotonic quantity related to the sharp gradient estimate for $s$ on closed manifolds with positive Ricci curvature (cf.~{\cite[Theorem 4.11]{ManeaSGE}}).

The derivative of $V$ may be directly computed using the coarea formula at all regular values $r$ of $s$ (cf.~{\cite[Equation (4.22)]{ManeaSGE}}) -- which form a full measure subset of the range of $s$ by Sard's theorem -- and we have
\begin{equation} \label{eq: derivative of V function}
    V'(r) = r^{-n - 1} \left( -n \int_{s \leq r} \left( \|\nabla s\|^2 + \frac{k}{4} s^2 \right) \left( \|\nabla s\|^2 - \frac{k}{4} s^2 \right) + r \int_{s = r} \left( \|\nabla s\|^2 + \frac{k}{4} s^2 \right) \|\nabla s\| \right).
\end{equation}
With this, we may now directly prove \eqref{intro-eq: monotonicity of r squared times V minus volume of unit ball} and deduce its corresponding rigidity statement; the argument is a positive-curvature extension of the one from {\cite[Proof of Theorem 1.6]{ParkHessianMonotonicity}}.

\begin{proof}[Proof of Theorem \ref{intro-th: monotonicity of r squared times V minus volume of unit ball}]
Since the set of regular values of $s$ is an open and full measure subset of its range (by Sard's theorem and compactness of $M$) and $V$ is differentiable at all such values, it suffices to compute the desired derivative (cf.~\eqref{intro-eq: monotonicity of r squared times V minus volume of unit ball}) at all regular values $r$ of $s$, with its corresponding rigidity statement.
Thus, fix a regular value $r > 0$ of $s$.

On the one hand, integration by parts tells us that
\[
\int_{s \leq r} \left\langle \nabla \left( \|\nabla s\|^2 + \frac{k}{4} s^2 \right), \nabla s^2 \right\rangle = - \int_{s \leq r} \Delta s^2 \cdot \left( \|\nabla s\|^2 + \frac{k}{4} s^2 \right) + \int_{s = r} \left( \|\nabla s\|^2 + \frac{k}{4} s^2 \right) \left\langle \nabla s^2, \frac{\nabla s}{\|\nabla s\|} \right\rangle.
\]
Using that $\Delta s^2 = 2n(\|\nabla s\|^2 - ks^2/4)$ (cf.~\eqref{eq: Delta s squared}) and combining the resulting equation with \eqref{eq: derivative of V function}, it follows that
\begin{equation} \label{eq: derivative of V in terms of hessian quantity part 1}
    2 r^{n + 1} V'(r) = \int_{s \leq r} \left\langle \nabla \left( \|\nabla s\|^2 + \frac{k}{4} s^2 \right), \nabla s^2 \right\rangle.
\end{equation}

On the other hand, since $d \|\nabla s\|^2 = 2 \hess{s}(\nabla s, \cdot)$, the product rule allows us to compute
\allowdisplaybreaks
\begin{align*}
    \left\langle \nabla \left( \|\nabla s\|^2 + \frac{k}{4} s^2 \right), \nabla s^2 \right\rangle &= 4s \hess{s} (\nabla s, \nabla s) + k s^2 \|\nabla s\|^2 \\
    &= 2 \hess{s^2} (\nabla s, \nabla s) - 4 \|\nabla s\|^4 + ks^2 \|\nabla s\|^2 \\
    &= 2 (\hess{s}^2 + ks^2g - 2g)(\nabla s, \nabla s) + 4 \|\nabla s\|^2 \left(1 - \|\nabla s\|^2 - \frac{k}{4} s^2 \right),
\end{align*}
Combining this preceding equation with \eqref{eq: derivative of V in terms of hessian quantity part 1}, we obtain that
\begin{equation} \label{eq: derivative of V in terms of hessian quantity part 2}
    2 r^{n + 1} V'(r) = 2 \int_{s \leq r} (\hess{s}^2 + ks^2g - 2g) (\nabla s, \nabla s) + 4 \int_{s \leq r} \|\nabla s\|^2 \left(1 - \|\nabla s\|^2 - \frac{k}{4} s^2 \right).
\end{equation}

Now, {\cite[Corollary 4.9]{ManeaSGE}} tells us that
\[
r^{-n} \int_{s \leq r} \left( \|\nabla s\|^2 - \frac{k}{4} s^2 \right) + \frac{k}{4} \int_{s \leq r} G = \frac{\omega_{n - 1}}{n}.
\]
Subtracting $V(r)$ from this preceding identity (cf.~\eqref{eq: definition of V function}), we get that
\allowdisplaybreaks
\begin{align*}
    \frac{\omega_{n - 1}}{n} - V(r) &= r^{-n} \int_{s \leq r} \left(1 - \|\nabla s\|^2 - \frac{k}{4} s^2 \right) \left( \|\nabla s\|^2 - \frac{k}{4} s^2 \right) + \frac{k}{4} \int_{s \leq r} \left(1 - \|\nabla s\|^2 - \frac{k}{4} s^2 \right) G \\
    &= r^{-n} \int_{s \leq r} \left(1 - \|\nabla s\|^2 - \frac{k}{4} s^2 \right) \left( \|\nabla s\|^2 - \frac{k}{4} s^2 + \frac{k}{4} r^n G \right) \\
    &= r^{-n} \int_{s \leq r} \left(1 - \|\nabla s\|^2 - \frac{k}{4} s^2 \right) \left( \|\nabla s\|^2 - \frac{k}{4} s^2 \left(1 - \frac{r^n}{s^n} \right) \right),
\end{align*}
where the last equality follows because $G \equiv s^{2 - n}$, and thus
\begin{equation} \label{eq: derivative of V in terms of hessian quantity part 3}
    \begin{split}
        \frac{\omega_{n - 1}}{n} - V(r) &= r^{-n} \int_{s \leq r} \left(1 - \|\nabla s\|^2 - \frac{k}{4} s^2 \right) \|\nabla s\|^2 \\
        &\quad \, - \frac{k}{4} r^{-n} \int_{s \leq r}  s^2 \left( 1 - \frac{r^n}{s^n} \right) \left( 1 - \|\nabla s\|^2 - \frac{k}{4} s^2 \right).
    \end{split}
\end{equation}
The desired identity (cf.~\eqref{intro-eq: monotonicity of r squared times V minus volume of unit ball}) now follows by combining \eqref{eq: derivative of V in terms of hessian quantity part 3} with \eqref{eq: derivative of V in terms of hessian quantity part 2} and dividing the result by $2 r^{n - 1}$.

For the rigidity statement, suppose that the right-hand side of \eqref{intro-eq: monotonicity of r squared times V minus volume of unit ball} is zero for some regular value $r$ of $s$.
It then follows that $\hess{s}^2(\nabla s, \nabla s) + k s^2 \|\nabla s\|^2 = 2 \|\nabla s\|^2$ on $\{s \leq r\}$, and so the rigidity statement of Theorem \ref{intro-th: sharp hessian estimate} implies that $(M, g)$ must be isometric to the round $n$-sphere of constant curvature equal to $k$.
\end{proof}
\section{Applications} \label{sec: applications}

In this last section, we will present some applications of the sharp Hessian inequality from Theorem \ref{intro-th: sharp hessian estimate}.
As in the previous sections, suppose that $(M, g)$ is a closed, connected, Einstein manifold of dimension $n \geq 3$, with sectional curvature bounded below by $k > 0$, fix a point $p \in M$, let $G_p$ be Green's function for the operator $L := - \Delta + n(n - 2)k/4$, with singularity at $p$, and put $s_p := G_p^{1/(2 - n)}$.

Our first application is a sharp comparison inequality for $(M, g)$ using two distinct Green's functions; let us describe the situation in the model case when $(M, g) \simeq \S^n_k$ -- the round $n$-sphere of constant curvature equal to $k$.
As discussed in the previous sections, in this case, the operator $L$ is the conformal Laplacian for $\S^n_k$, and one can directly compute (cf.~{\cite[Lemma 2.1]{ManeaSGE}}) that $s_p = 2 \snk(r_p/2)$, where $r_p$ is the distance to $p$ in $\S^n_k$ and $\snk$ is the curvature-adapted sine function.
In particular, $\max_{\S^n_k} s_p^2 = 4/k$, and $s_p^2$ has a unique global maximum, the antipodal point of $p$, which we denote by $-p$.
If $r_{-p}$ denotes the distance to $-p$ in $\S^n_k$, then we also know that 
\begin{equation} \label{eq: sum of distances between antipodal points on the sphere}
    r_p + r_{-p} = \pi/\sqrt{k}.
\end{equation}
Since, by {\cite[Lemma 2.1]{ManeaSGE}}, $s_{-p} = 2 \snk(r_{-p}/2)$ (where $s_{-p} := G_{-p}^{1/(2 - n)}$ and $G_{-p}$ is Green's function for $L$ with singularity at $-p$), \eqref{eq: sum of distances between antipodal points on the sphere} tells us that $s_{-p} = 2 k^{-1/2} \csk(r_p/2)$, where $\csk := \snk'$ is the curvature-adapted cosine function.
Thus, on $\S^n_k$, we have that
\begin{equation} \label{eq: sum of antipodal points greens function on the sphere}
    s_p^2 + s_{-p}^2 \equiv \frac{4}{k}.
\end{equation}

For a general manifold $(M, g)$ as before, {\cite[Proposition 2.2]{ManeaSGE}} tells us that $s_p \leq 2 \snk(\dist(p, \cdot)/2)$, and so, in particular, by diameter comparison, we get that $s_p^2 \leq 4/k$.
Our first application is an inequality in the spirit of \eqref{eq: sum of antipodal points greens function on the sphere}, which allows us to compare $(M, g)$ to $\S^n_k$ using two distinct Green's functions instead of just one, whose singularities will act as ``comparison antipodal points''; similar such ideas have also been recently used in {\cite{ParkRigidityPositiveCurvature}} to prove a Weyl tensor rigidity statement for the round four-sphere.

In the notation from \eqref{eq: sum of antipodal points greens function on the sphere}, a special property that $\S^n_k$ possesses is that the (unique) maximum point of $s_{-p}$ is $p$.
In general, $s_p$ may have more than one maximum point, and if $q \in M \setminus \{p\}$ is such a point, it is not necessarily the case that $p$ is also a maximum point for $s_q$, where $s_q := G_q^{1/(2 - n)}$ and $G_q$ is Green's function for $L$ with singularity at $q$; however, two such points always exist by the compactness of $M$.
Our first application is a general inequality -- similar to \eqref{eq: sum of antipodal points greens function on the sphere} -- that accounts for this.

\begin{theorem} \label{th: pole-antipodal point inequality}
Let $(M, g)$ be a closed, connected, Einstein manifold of dimension $n \geq 3$, with sectional curvature bounded below by $k > 0$.
Fix a point $p \in M$, let $G_p$ be Green's function for the operator $L := - \Delta + n(n - 2)k/4$, with singularity at $p$, and put $s_p := G_p^{1/(2 - n)}$.
Denote by $m > 0$ the maximal value of $s_p$ on $M$, let $q \in M \setminus \{p\}$ be so that $s_p(q) = m$, and let $G_q$ be Green's function for $L$ with singularity at $q$.

Then, for $s_q := G_q^{1/(2 - n)}$, we have
\begin{equation} \label{eq: pole-antipodal point inequality}
    s_p^2 + \left( \frac{k m^2}{2} - 1 \right) s_q^2 \leq m^2,
\end{equation}
and equality holds at some \emph{point} in $M \setminus \{p,q\}$ if and only if $(M, g)$ is isometric to the round $n$-sphere of constant curvature equal to $k$.
\end{theorem}

\begin{proof}
Recall that, since $L$ is a symmetric operator, the Green kernel of $L$ is also symmetric, and so $s_p(q) = s_q(p) = m$.
This immediately implies, in particular, that \eqref{eq: pole-antipodal point inequality} holds at the points $p$ and $q$ if we would define $s_p(p) := 0 =: s_q(q)$.

Let us now proceed with the proof of \eqref{eq: pole-antipodal point inequality} on $M \setminus \{p, q\}$.
We remark that we may assume, without loss of generality, that $m^2 \geq 2/k$.
Indeed, if $m^2 < 2/k$, then $(km^2/2) - 1 < 0$, which clearly implies that \eqref{eq: pole-antipodal point inequality} holds on $M \setminus \{p, q\}$ as $s_p^2 \leq m^2$ by definition of $m$.
Define the function $u := m^2 - s_p^2 - (km^2/2 - 1) s_q^2$; we need to show that $u \geq 0$ on $M \setminus \{p, q\}$, with equality holding at some point if and only if $(M, g)$ is isometric to the round $n$-sphere of curvature $k$.
Observe that, by the asymptotics of $s_p^2$ and $s_q^2$ near $p$ and $q$, respectively (cf.~Corollary \ref{appendix-cllr: sharp c2 asymptotics for s near singularity}), both $s_p^2$ and $s_q^2$ are $C^2$ on $M$ when extended as $s_p(p) := 0 =: s_q(q)$.
Moreover, a direct computation shows that
\allowdisplaybreaks
\begin{align} \label{align: convexity of pole-antipodal point function}
    \hess{u} + k u g &= - \hess{s_p^2} - \left( \frac{km^2}{2} - 1 \right) \hess{s_q^2} + k \left( m^2 - s_p^2 - \left( \frac{km^2}{2} - 1 \right) s_q^2 \right) g \nonumber \\
    &= - \left( \hess{s_p^2} + ks_p^2 g - 2g \right) - \left( \frac{km^2}{2} - 1 \right) \left( \hess{s_q^2} + ks_q^2 g - 2g \right),
\end{align}
and the right-hand side above is non-negative by Theorem \ref{intro-th: sharp hessian estimate} and because $km^2 - 2 \geq 0$.

We now proceed using a similar idea as in the proof of Obata's rigidity theorem for the sphere (cf.~{\cite{ObataSphericalRigidity}}): we will use \eqref{align: convexity of pole-antipodal point function} to solve the differential equation for $u$ along every minimizing geodesic starting at $q$.
Specifically, let $x \in M \setminus \{p, q\}$, and let $\gamma : [0, \ell] \to M$ be a unit-speed, minimizing geodesic segment from $q$ to $x$, where $\ell := \dist(q, x)$; by the diameter comparison theorem, $\ell \leq \pi/\sqrt{k}$.
Equation \eqref{align: convexity of pole-antipodal point function} tells us that
\begin{equation} \label{eq: ode for pole-antipodal point function}
    \frac{\dd^2}{\dd t^2} (u \circ \gamma) + k (u \circ \gamma) = f,
\end{equation}
where $f$ is the non-negative function defined as
\begin{equation} \label{eq: hessian of pole-antipodal point function}
    f(t) := - \left( \hess{s_p^2}(\gamma'(t), \gamma'(t)) + ks_p(\gamma(t))^2 - 2 \right) - \left( \frac{km^2}{2} - 1 \right) \left( \hess{s_q^2}(\gamma'(t), \gamma'(t)) + ks_q(\gamma(t))^2 - 2 \right).
\end{equation}
Since we also have that $u(\gamma(0)) = (u \circ \gamma)'(0) = 0$, we may solve the differential equation \eqref{eq: ode for pole-antipodal point function} with these initial conditions and conclude that $u(\gamma(t)) = \int_0^t \snk(t - \tau) f(\tau) \dd \tau$ for all $t \in [0, \ell]$.
Hence,
\begin{equation} \label{eq: expression of pole-antipodal map in terms of integral of its hessian}
    u(x) \equiv u(\gamma(\ell)) = \int_0^\ell \snk(\ell - \tau) f(\tau) \dd \tau,
\end{equation}
and the right-hand side above is non-negative because $f(\tau) \geq 0$ and $\snk(\ell - \tau) \geq 0$ for all $\tau \in [0, \ell]$.

For the rigidity statement, suppose that $u(x) = 0$ for some $x \in M \setminus \{p, q\}$.
Using \eqref{eq: expression of pole-antipodal map in terms of integral of its hessian} and recalling that $\snk > 0$ on $(0, \pi/\sqrt{k})$, it follows that $f \equiv 0$ on $(0, \ell)$ where $\ell := \dist(q, x)$ and $f$ is defined as in \eqref{eq: hessian of pole-antipodal point function}, and so, in particular, $\hess{s_p^2} (\gamma', \gamma') + k (s_p \circ \gamma)^2 - 2 = 0$ on $(0, \ell)$.
The rigidity statement of Theorem \ref{intro-th: sharp hessian estimate} then implies that $(M, g)$ is isometric to the round $n$-sphere of constant curvature equal to $k$.
\end{proof}

Using the compactness of $M$ and the symmetry of the Green kernel of $-\Delta + n(n - 2)k/4$, we may now use Theorem \ref{th: pole-antipodal point inequality} for two points $p$ and $q$ for which $p$ and $q$ are maximum points of $s_q$ and $s_p$ -- which act as ``comparison antipodal points'' -- to obtain a sharp inequality similar to \eqref{eq: sum of antipodal points greens function on the sphere}; this is the content of Theorem \ref{intro-th: pole-antipodal point pointwise inequality for greens functions}.

\begin{proof}[Proof of Theorem \ref{intro-th: pole-antipodal point pointwise inequality for greens functions}]
Let $\mathcal{G}$ be the Green kernel of the operator $L := -\Delta + n(n - 2)k/4$, so that for every $p \in M$, $G_p := \mathcal{G}(p, \cdot)$ is Green's function for $L$ with singularity at $p$; denote by $\mathcal{S} := \mathcal{G}^{1/(2 - n)}$.
Since $L$ is a symmetric and elliptic differential operator, we know that $\mathcal{G}$ is a symmetric and smooth function away from the diagonal of $M \times M$, and the asymptotics of $\mathcal{G}$ near the diagonal of $M \times M$ (cf.~{\cite[Section 4.2]{AubinNonlinearProblemsRG}} or {\cite[Theorem A.1]{ManeaSGE}}) tell us that $\mathcal{S}$ extends to a continuous function on $M \times M$ by setting $\mathcal{S}(p, p) := 0$ for all $p \in M$.

Now, by compactness of $M$ and continuity, the function $\mathcal{S}$ attains its maximum $m > 0$ at some point $(p, q) \in M \times M$, and we necessarily have that $p \neq q$ because $\mathcal{S}$ is positive away from the diagonal of $M \times M$.
In this case, since $s_p = \mathcal{S}(p, \cdot)$ and $s_q = \mathcal{S}(q, \cdot)$, it then follows that $s_p(q) = s_q(p) = m$ and that $m$ is the maximum of both $s_p$ and $s_q$.
Theorem \ref{th: pole-antipodal point inequality} then tells us that
\begin{equation} \label{eq: final equation for pole-antipodal point pointwise inequality}
    s_p^2 + \left( \frac{k m^2}{2} - 1 \right) s_q^2 \leq m^2 \quad \text{ and } \quad s_q^2 + \left( \frac{k m^2}{2} - 1 \right) s_p^2 \leq m^2.
\end{equation}
Summing up these two inequalities and grouping the common terms gives us the inequality $s_p^2 + s_q^2 \leq 4/k$.

For the rigidity statement, suppose that $s_p^2 + s_q^2 = 4/k$ at some point of $M \setminus \{p,q\}$.
It follows that both inequalities from \eqref{eq: final equation for pole-antipodal point pointwise inequality} are equalities at that said point, and thus $(M, g)$ must be isometric to the round $n$-sphere of constant curvature equal to $k$ by the rigidity part of Theorem \ref{th: pole-antipodal point inequality}.
\end{proof}

For the second application, we may suppress the dependence of the singularity in the notation for our Green's function, and denote by $s := G^{1/(2 - n)}$, where $G$ is Green's function for $-\Delta + n(n - 2)k/4$ with singularity at $p$.
As previously mentioned, {\cite[Theorem 1.1]{ManeaSGE}} (cf.~\eqref{intro-eq: sharp gradient estimate}) shows that $\|\nabla s\|^2 + ks^2/4 \leq 1$, with equality holding at \emph{some} point if and only if $(M, g)$ is isometric to the round $n$-sphere of curvature $k$.
The proof of this shows that if $\|\nabla s\|^2 + ks^2/4 = 1$ at some point of $M \setminus \{p\}$, then $\hess{s^2} = (2 - ks^2)g$ on $M \setminus \{p\}$; the general theory for the rigidity of warped products (cf.~{\cite[Section 1]{CheegerColdingWarpedProducts}} or {\cite[Corollary 4.3.4]{PetersenRG}}) tells us that this latter equation implies the desired rigidity.

By the theory for the rigidity of warped products from {\cite{CheegerColdingWarpedProducts}}, we also know that if $\hess{s^2} + ks^2 g - 2g$ is close to zero, then $(M, g)$ is Gromov-Hausdorff close to the round $n$-sphere of curvature $k$ (here, ``closeness'' for $\hess{s^2} + ks^2 g - 2g$ may be measured in various ways, such as pointwise or, as in {\cite{CheegerColdingWarpedProducts}}, in an averaged $L^2$ sense).
Recently (cf.~{\cite{ParkRigidityPositiveCurvature}}), the authors have shown that if the gradient deficit $\|\nabla s\|^2 + ks^2/4 - 1$ is close to $0$ in the sense that the averaged integral of $(1 - \|\nabla s\|^2 - ks^2/4) G$ over $M$ is close to zero, then $(M, g)$ is close to the round $n$-sphere of curvature $k$ in the Gromov-Hausdorff sense.
Our next proposition shows that, in our setting of Einstein manifolds with positive sectional curvature, from this rigidity point of view, the closeness of the gradient deficit to zero is equivalent to the ``Hessian deficit'' $\hess{s^2} + ks^2 g - 2g$ being close to zero.

\begin{proposition} \label{prop: hessian warped product quantity bounded by gradient quantity}
Let $(M, g)$ be a closed, connected, Einstein manifold of dimension $n \geq 3$, with sectional curvature bounded below by $k > 0$, and let $G$ be Green's function for the operator $-\Delta + n(n - 2)k/4$, with singularity at a fixed point $p \in M$; put $s := G^{\frac{1}{2 - n}}$.

Then, on $M \setminus \{p\}$, we have
\[
4n \left( \|\nabla s\|^2 + \frac{k}{4} s^2 - 1 \right)^2 \leq \|\!\hess{s^2} + ks^2g - 2g\|^2 \leq 4n^2 \left( \|\nabla s\|^2 + \frac{k}{4} s^2 - 1 \right)^2.
\]
\end{proposition}

\begin{proof}
By Theorem \ref{intro-th: sharp hessian estimate}, we know that $\hess{s^2} + ks^2g - 2g \leq 0$ in the sense of quadratic forms on $M \setminus \{p\}$.
A simple linear algebra result then tells us that $\|\!\hess{s^2} + ks^2g - 2g\| \leq - \tr(\hess{s^2} + ks^2g - 2g)$.
By \eqref{eq: Delta s squared}, we may directly compute 
\[
\tr(\hess{s^2} + ks^2g - 2g) = \Delta s^2 + nks^2 - 2n = 2n \left(\|\nabla s\|^2 + \frac{k}{4} s^2 - 1 \right),
\]
and thus
\[
\|\!\hess{s^2} + ks^2g - 2g\|^2 \leq (\tr(\hess{s^2} + ks^2g - 2g))^2 = 4n^2 \left(\|\nabla s\|^2 + \frac{k}{4} s^2 - 1 \right)^2.
\]

For the other inequality, we may first compute
\[
\hess{s^2} + ks^2 g - 2g = \tracelessHess{s^2} + \left( \frac{\Delta s^2}{n} + ks^2 - 2 \right) g = \tracelessHess{s^2} + 2 \left( \|\nabla s\|^2 + \frac{k}{4} s^2 - 1 \right) g,
\]
where the last equality follows from \eqref{eq: Delta s squared}.
Taking norms in the above identity, it follows that
\[
4n \left( \|\nabla s\|^2 + \frac{k}{4} s^2 - 1 \right)^2 \leq 4n \left( \|\nabla s\|^2 + \frac{k}{4} s^2 - 1 \right)^2 + \|\!\tracelessHess{s^2}\|^2 = \|\!\hess{s^2} + ks^2g - 2g\|^2,
\]
which completes the proof of this proposition.
\end{proof}

\appendix

\section{Asymptotics of Green's functions near their singularity on closed manifolds} \label{appendix-sec: asymptotics of Greens function near the pole on closed manifolds}

In this appendix, we will discuss the $C^2$ asymptotics of Green's functions for the operator $-\Delta + n(n - 2)k/4$ near their singularity, where $k > 0$.
Our proof is a simple adaptation of the argument from {\cite[Proof of Theorem 2.1 in Appendix A]{HessianAsymptoticsGreensFunction}}, as our operator is not the Laplacian, but a Schr\"{o}dinger operator with constant potential.
In addition, since we are dealing only with the usual Laplacian (which corresponds to the $p = 2$ case in {\cite[Theorem 2.1]{HessianAsymptoticsGreensFunction}}) and since the $C^1$ asymptotics of Green's function have already been proved in {\cite[Theorem A.1]{ManeaSGE}}, the proof of the $C^2$ asymptotics becomes much more streamlined.

\begin{theorem} \label{appendix-th: hessian asymptotics of greens function in positive curvature}
Let $(M, g)$ be a closed, connected Riemannian manifold of dimension $n \geq 3$, let $p \in M$ be a fixed point, and let $k > 0$ be a constant.
Let $G$ be Green's function for the operator $-\Delta + n(n - 2)k/4$ with singularity at $p$, and put $r := \dist(p, \cdot)$.

Then, as $r \to 0$, we have that $G = r^{2 - n} + o_2(r^{2 - n})$, in the sense that
\begin{equation} \label{appendix-eq: sharp c1 asymptotics of G near the pole}
    G = r^{2 - n} + o(r^{2 - n}), \qquad \nabla G = (2 - n) r^{1 - n} \nabla r + o(r^{1 - n})
\end{equation}
and
\begin{equation} \label{appendix-eq: sharp hessian asymptotics of G near the pole}
    \hess{G} = \hess{r^{2 - n}} + o(r^{-n}) = (n - 2) r^{-n} (n \, dr \otimes dr - g) + o(r^{-n}).
\end{equation}
\end{theorem}

\begin{proof}
As mentioned before, \eqref{appendix-eq: sharp c1 asymptotics of G near the pole} is proved in {\cite[Theorem A.1]{ManeaSGE}}; thus, we only need to prove \eqref{appendix-eq: sharp hessian asymptotics of G near the pole}.
We proceed as in {\cite[Appendix A]{HessianAsymptoticsGreensFunction}} and define, for $\lambda > 0$, the map $T_\lambda : T_p M \to M$, $T_\lambda(v) := \exp_p(\lambda v)$ and the metric $g_\lambda := T_\lambda^*(\lambda^{-2} g)$ on $T_p M$.
We then know that $g_\lambda \to g|_p$ as $\lambda \downarrow 0$ in the smooth topology of $T_p M \setminus \{0\}$ (with respect to the $g|_p$ metric), as this is just the computation of the tangent cone of $M$ at $p$.
We also define the rescaled Green's function $G_\lambda$ on $T_p M \setminus \{0\}$ as $G_\lambda := \lambda^{n - 2} (G \circ T_\lambda)$.
A direct computation shows that $\Delta_{g_\lambda} G_\lambda = \lambda^n (\Delta_g G) \circ T_\lambda$, where $\Delta_{g_\lambda}$ is the Laplacian in the $g_\lambda$ metric on $T_p M$, and thus, on $T_p M \setminus \{0\}$, $G_\lambda$ satisfies the equation
\begin{equation} \label{appendix-eq: equation for rescaled greens function}
    \Delta_{g_\lambda} G_\lambda = \frac{n(n - 2)k}{4} \lambda^2 G_\lambda.
\end{equation}

Let $G_0$ be the minimal, positive Green's function for the Laplacian in the metric $g|_p$ on $T_p M$ with singularity at $0 \in T_p M$; this means that $G_0(v) = \|v\|^{2 - n}$ for $v \in T_p M \setminus \{0\}$.
Since $r(\exp(v)) = \|v\|$ for $v \in T_p M \setminus \{0\}$ with $\|v\|$ small enough, \eqref{appendix-eq: sharp c1 asymptotics of G near the pole} implies that $G_\lambda \to G_0$ locally uniformly on $T_p M \setminus \{0\}$ as $\lambda \to 0$.
We claim that $G_\lambda \to G_0$ in the $C^2$ topology of $T_p M \setminus \{0\}$ as $\lambda \to 0$ (with respect to $g|_p$).
Let us prove this claim.
By \eqref{appendix-eq: equation for rescaled greens function}, on $T_p M \setminus \{0\}$, 
\[
\left( \Delta_{g_\lambda} - \frac{n(n - 2)k}{4} \lambda^2 \right) \left( G_\lambda - G_0 \right) = - \left( \Delta_{g_\lambda} - \frac{n(n - 2)k}{4} \lambda^2 \right) G_0.
\]
As $g_\lambda \to g|_p$ and $n(n - 2)k \lambda^2/4 \to 0$ smoothly on $T_p M \setminus \{0\}$ as $\lambda \to 0$ (with respect to $g|_p$), the right-hand side above converges to zero as $\lambda \to 0$ in the smooth topology of $T_p M \setminus \{0\}$.
Elliptic regularity estimates then imply that $G_\lambda \to G_0$ in the $C^2$ topology of $T_p M \setminus \{0\}$ as $\lambda \to 0$ (with respect to $g|_p$), which finishes the proof of the claim.

With the above claim, we proceed as follows.
Denote by $\rho := \|\cdot\|$ in $T_p M$, so $G_0 = \rho^{2 - n}$.
The chain rule tells us that $\hess_{g|_p}{G_0} = (n - 2) \rho^{-n} (n \, d \rho \otimes d \rho - g|_p)$, where $\hess{g|_p}$ denotes the Hessian in the metric $g|_p$ on $T_p M$.
Since $G_\lambda \to G_0$ in the $C^2$ topology of $T_p M \setminus \{0\}$ and since $g_\lambda \to g|_p$ smoothly as $\lambda \to 0$ (with respect to $g|_p$), this implies that, on every compact subset $K \subset T_p M \setminus \{0\}$, we have
\begin{equation} \label{appendix-eq: hessian of rescaled greens function part 1}
    \hess_{g_\lambda} G_\lambda = (n - 2) \rho^{-n} (n \, d \rho \otimes d \rho - g_\lambda) + o(1)
\end{equation}
as $\lambda \to 0$, and the convergence is uniform on $K$ and with respect to the metric $g|_p$.
However, since $g_\lambda \to g|_p$ as $\lambda \to 0$, $g|_p$ and $g_\lambda$ define norms on $T_p M$ which are uniformly equivalent for $\lambda > 0$ small enough (\ie there is a constant $c > 0$ so that $c^{-1} g_\lambda \leq g|_p \leq c g_\lambda$ for all $\lambda > 0$ small enough); together with \eqref{appendix-eq: hessian of rescaled greens function part 1} this implies that,
\begin{equation} \label{appendix-eq: hessian of rescaled greens function part 2}
    \|\!\hess_{g_\lambda} G_\lambda - (n - 2) \rho^{-n} (n \, d \rho \otimes d \rho - g_\lambda) \|_{g_\lambda} \to 0
\end{equation}
as $\lambda \to 0$ uniformly on every compact subset of $T_p M \setminus \{0\}$.

Now, since $g_\lambda \equiv T_\lambda^*(\lambda^{-2} g)$, we may also explicitly compute that $\hess_{g_\lambda} G_\lambda = \lambda^{n - 2} T_\lambda^* (\hess_g G)$.
Similarly, $r \circ T_\lambda = \lambda \rho$ and $T_\lambda^* (dr) = \lambda \, d\rho$ near $p$, and so 
\begin{equation} \label{appendix-eq: hessian of rescaled greens function part 3}
    T_\lambda^* ((n - 2) r^{-n} (n \, dr \otimes dr - g)) = \lambda^{2 - n} (n - 2) \rho^{-n} (n \, d\rho \otimes d \rho - g_\lambda).
\end{equation}
Combining \eqref{appendix-eq: hessian of rescaled greens function part 1} with \eqref{appendix-eq: hessian of rescaled greens function part 2} and \eqref{appendix-eq: hessian of rescaled greens function part 3}, we obtain that, on every compact subset of $T_p M \setminus \{0\}$, we have that
\[
\lambda^{n - 2} \|T_\lambda^* (\hess{G} - (n - 2) r^{-n} (n \, dr \otimes dr - g))\|_{g_\lambda} \to 0
\]
uniformly as $\lambda \to 0$.
However, as $T_\lambda^*g = \lambda^2 g_\lambda$, the above may be rewritten as
\begin{equation} \label{appendix-eq: hessian of rescaled greens function final}
    \lambda^n \left( \|\!\hess{G} - (n - 2) r^{-n} (n \, dr \otimes dr - g)\|_g \circ T_\lambda \right) \to 0
\end{equation}
as $\lambda \to 0$ uniformly on every compact subset $K \subset T_p M \setminus \{0\}$. 

Let us now use \eqref{appendix-eq: hessian of rescaled greens function final} to complete the proof.
Observe that, if $x \in M \setminus \{p\}$ is close enough to $p$ so we may write $x = \exp_p(v)$ for a unique $v \in T_p M \setminus \{0\}$, then we also have that $\|v\| = \dist(p, x) \equiv r(x)$, and thus $x = \exp_p( r(x) \cdot v/r(x)) = T_{r(x)}(v/r(x))$.
It follows that for every $x \in M \setminus \{p\}$ close enough to $p$, we can write $x = T_{r(x)}(v_x)$ for some $v_x \in T_p M$ with $\|v_x\| = 1$.
Thus, taking the compact subset $K \subset T_p M \setminus \{0\}$ to be the unit sphere in the $g|_p$ metric in $T_p M$ and replacing $\lambda$ with $r(x)$, \eqref{appendix-eq: hessian of rescaled greens function final} allows us to conclude that
\[
r(x)^n \left( \|\!\hess{G}|_x - (n - 2) r(x)^{-n} (n \, dr|_x \otimes dr|_x - g|_x)\|_g\right) \to 0
\]
as $x \to p$, which is precisely \eqref{appendix-eq: sharp hessian asymptotics of G near the pole}.
\end{proof}

With the above theorem, we are also ready to state and prove the sharp $C^2$ asymptotics of the function $s^2 \equiv G^{2/(2 - n)}$ that will be needed in the proof of Theorem \ref{intro-th: sharp hessian estimate}.

\begin{corollary} \label{appendix-cllr: sharp c2 asymptotics for s near singularity}
In the setting and notation from Theorem \ref{appendix-th: hessian asymptotics of greens function in positive curvature}, if we denote by $s := G^{\frac{1}{2 - n}}$, then, near $p$, we have
\begin{equation} \label{appendix-eq: sharp c1 asymptotics of s near the pole}
    s = \dist(p, \cdot) + o(\dist(p, \cdot)), \qquad \nabla s = \nabla \dist(p, \cdot) + o(1),
\end{equation}
and
\begin{equation} \label{appendix-eq: sharp hessian asymptotics of s squared near the pole}
    \hess{s^2} = 2g + o(1).
\end{equation}
\end{corollary}

\begin{proof}
Recall that $s$ is well-defined because $G > 0$ by the strong maximum principle.
The asymptotics of $s$ follow immediately from the chain rule and Theorem \ref{appendix-th: hessian asymptotics of greens function in positive curvature}.
Indeed, to see that \eqref{appendix-eq: sharp c1 asymptotics of s near the pole} holds, if we denote by $r := \dist(p, \cdot)$, then, by \eqref{appendix-eq: sharp c1 asymptotics of G near the pole}, we have
\[
s = G^{\frac{1}{2 - n}} = \left( r^{2 - n}(1 + o(1)) \right)^{\frac{1}{2 - n}} = r(1 + o(1)),
\]
and
\[
\nabla s = \frac{1}{2 - n} G^{\frac{n - 1}{2 - n}} \nabla G = \frac{1}{2 - n} \cdot r^{n - 1}(1 + o(1)) \cdot \left( (2 - n) r^{1 - n} \nabla r + o(r^{1 - n}) \right) = \nabla r + o(1).
\]
To prove \eqref{appendix-eq: sharp hessian asymptotics of s squared near the pole}, we may first compute, with the chain rule, that
\[
\hess{s^2} \equiv \hess{G^{\frac{2}{2 - n}}} = \frac{2}{2 - n} G^{\frac{n}{2 - n}} \hess{G} + \frac{2n}{(n - 2)^2} G^{\frac{2n - 2}{2 - n}} dG \otimes dG,
\]
and then use \eqref{appendix-eq: sharp c1 asymptotics of G near the pole} and \eqref{appendix-eq: sharp hessian asymptotics of G near the pole} to deduce that
\allowdisplaybreaks
\begin{align*}
    \frac{2}{2 - n} G^{\frac{n}{2 - n}} \hess{G} &= \frac{2}{2 - n} \cdot r^n(1 + o(1)) \cdot \left( (n - 2) r^{-n} (n dr \otimes dr - g) + o(r^{-n}) \right) \\
    &= 2g - 2n \, dr \otimes dr + o(1),
\end{align*}
and
\allowdisplaybreaks
\begin{align*}
    \frac{2n}{(n - 2)^2} G^{\frac{2n - 2}{2 - n}} dG \otimes dG &= \frac{2n}{(n - 2)^2} \cdot r^{2n - 2}(1 + o(1)) \cdot \left( (2 - n) r^{1 - n} dr + o(r^{1 - n}) \right) \otimes \left( (2 - n) r^{1 - n} dr + o(r^{1 - n}) \right) \\
    &= 2n \, dr \otimes dr + o(1),
\end{align*}
from which it immediately follows that $\hess{s^2} = 2g + o(1)$.
\end{proof}
\section{The Hessian inequality on non-compact manifolds} \label{appendix-sec: hessian inequalilty on noncompact manifolds}

In this section, we discuss the Hessian inequality for non-parabolic manifolds with non-negative sectional curvature and parallel Ricci curvature -- which is the setting similar to that from {\cite{ParkHessianEstimate}} -- and show that the sharp analogue does not hold in general.
We employ the notation from {\cite{ParkHessianEstimate}} (which is slightly different than the one from \eqref{intro-eq: hessian inequality for non-parabolic manifolds}).
Let $(M, g)$ be a complete, non-parabolic Riemannian manifold of dimension $n \geq 3$, with non-negative sectional curvature and parallel Ricci curvature; fix a point $p \in M$ and let $G$ be the minimal, positive Green's function for the Laplacian of $(M, g)$ with singularity at $p$, and put $b := G^{1/(2 - n)}$.
As mentioned at the end of Paragraph \ref{intro-subsec: sharp hessian inequality} in Section \ref{sec: intro} (cf.~\eqref{intro-eq: hessian inequality for non-parabolic manifolds}), {\cite{ParkHessianEstimate}} proved that $\hess{b^2} \leq Cg$ on $M \setminus \{p\}$ for a constant $C \geq 10$, provided the same inequality holds near $p$ and at infinity on $M$.

In the spirit of Theorem \ref{intro-th: sharp hessian estimate}, the sharp analogue in this setting would be the inequality 
\begin{equation} \label{appendix-eq: supposed sharp hessian inequality for non-parabolic manifolds}
    \hess{b^2} \leq 2g,
\end{equation}
with equality at some non-zero vector forcing $(M, g)$ to be isometric to $\R^n$; indeed, this would, for instance, trace to Colding's sharp gradient estimate -- $\|\nabla b\|^2 \leq 1$ (cf.~{\cite{ColdingNewMonotonicityFormulas}}),
However, we will see that this is not the case; in fact, the inequality $\hess{b^2} \leq 2g$ may fail arbitrarily close to the singularity $p$ (even though, by {\cite[Theorem 2.1]{HessianAsymptoticsGreensFunction}}, $\hess{b^2} = 2g + o(1)$ near $p$).

The idea behind this is to compute the second term in the asymptotic expansion of Green's function near its singularity (which will involve the Ricci and scalar curvatures) and to use the anisotropy of the Ricci tensor.
Since the de Rham decomposition theorem tells us that, locally, $(M, g)$ splits isometrically as a direct product between a Euclidean factor and other closed, Einstein manifolds, for flat factors with large enough dimension, the scalar curvature term will be arbitrarily small, and the defect between $\hess b^2$ and $2g$, when restricted to a totally geodesic surface, will essentially be equal to the Hessian of the function $-x^2 (x^2 + y^2)$ on $\R^2$. 
As the Hessian of this latter function has both a positive and a negative eigenvalue, one of its eigenvectors will be the direction in which $\hess{b^2}$ is greater than $2g$.

To make the above argument formal, we require the aforementioned sharper asymptotics of Green's functions near their singularities.
These were proved in {\cite[Appendix A]{LuescherGreensFunctionCurvatureAsymptotics}} for closed manifolds, but the argument also applies to non-compact manifolds as it is local; we present it in the form of the following theorem.

\begin{theorem} \label{appendix-th: curvature asymptotics of greens function near the pole}
Let $(M, g)$ be a complete, non-compact Riemannian manifold without boundary of dimension $n \geq 5$, which is also non-parabolic.
Fix a point $p \in M$ and let $G$ be the minimal, positive Green's function for the Laplacian of $(M, g)$ with singularity at $p$; denote by $r := \dist(p, \cdot)$.

Then, in normal coordinates $x = (x^1, \cdots, x^n)$ centred at $p$, we have that
\begin{equation} \label{appendix-eq: curvature asymptotics of greens function near the pole}
    G(x) = r^{2 - n} \left( 1 + \frac{1}{12} \ric_{ij}(p) x^i x^j + \frac{\scal(p)}{12(n - 4)} r^2 \right) + O_2(r^{5 - n}),
\end{equation}
as $x \to p$, where these asymptotics are understood in the same manner as in Theorem \ref{appendix-th: hessian asymptotics of greens function in positive curvature}.
\end{theorem}

Remark also that the expression $\ric_{ij}(p) x^i x^j$ is independent of the chosen normal coordinates centred at $p$.

\begin{proof}[Proof of Theorem \ref{appendix-th: curvature asymptotics of greens function near the pole}]
The pointwise asymptotics
\begin{equation} \label{appendix-eq: curvature pointwise asymptotics of greens function}
    G(x) = r^{2 - n} \left( 1 + \frac{1}{12} \ric_{ij}(p) x^i x^j + \frac{\scal(p)}{12(n - 4)} r^2 \right) + O(r^{5 - n})
\end{equation}
are proved in {\cite[Appendix A]{LuescherGreensFunctionCurvatureAsymptotics}} for Green's functions for Schr\"{o}dinger operators on closed manifolds, and the argument applies to our case as well as the computation is local and the proof allows for the Green's function of the Laplacian as well (defined modulo its kernel) (cf.~{\cite[Equation (A12)]{LuescherGreensFunctionCurvatureAsymptotics}}).
More precisely, with our normalization of Green's function (\ie $\Delta G = - (n - 2) \omega_{n - 1} \delta_p$), {\cite[Equations (A14) and (A15)]{LuescherGreensFunctionCurvatureAsymptotics}} imply that
\[
G(x) = (n - 2) \omega_{n - 1} \pi^{-n/2} r^{2 - n} \left( \frac{1}{4} \Gamma \left( \frac{n}{2} - 1 \right) c_0(x) + \frac{1}{16} \Gamma \left( \frac{n}{2} - 2 \right) c_1(x) \cdot r^2 \right) + O(r^{5 - n})
\]
where $c_0$ and $c_1$ are functions with asymptotics given by (cf.~{\cite[Equations (A24) and (A25)]{LuescherGreensFunctionCurvatureAsymptotics}})
\[
c_0(x) = 1 + \frac{1}{12} \ric_{ij}(p) x^i x^j + O(r^3), \qquad c_1(x) = \frac{\scal(p)}{6} + O(r).
\]
We then know that $\Gamma((n/2) - 1) = ((n - 2) \omega_{n - 1})^{-1} 4 \pi^{n/2}$, and since $n \geq 5$, we also have that $\Gamma((n/2) - 2) = 2 \Gamma((n/2) - 1)/(n - 4)$; these identities then immediately imply \eqref{appendix-eq: curvature pointwise asymptotics of greens function}.
The proof of the corresponding asymptotics for the gradient and Hessian of $G$ follows from the same blow-up argument as in the proof of Theorem \ref{appendix-th: hessian asymptotics of greens function in positive curvature} (which, in this case, is precisely the argument from {\cite[Proof of Theorem 2.1 in Appendix A]{HessianAsymptoticsGreensFunction}}).
\end{proof}

We are now ready to formalize the argument for the failure of the Hessian inequality \eqref{appendix-eq: supposed sharp hessian inequality for non-parabolic manifolds}.

\begin{theorem} \label{appendix-th: counter examples for non-compact case}
There exists a constant $\delta > 0$ so that the following holds.
Let $(M, g)$ be a complete, non-parabolic Riemannian manifold of dimension $n \geq 5$, with non-negative sectional curvature and parallel Ricci tensor.
Fix a point $p \in M$ and let $G$ be the minimal, positive Green's function for the Laplacian of $(M, g)$ with singularity at $p$; put $b = G^{1/(2 - n)}$.

If, for some unit vector $v \in T_p M$, we have that $\delta (n - 4) \ric (v, v) > \scal(p)$, then every neighbourhood of $p$ contains points $q$ and unit vectors $w \in T_q M$ for which $\hess{b^2} (w, w) > 2$. 
\end{theorem}

\begin{proof}
The de Rham decomposition theorem implies that, near $p$, $(M, g)$ is isometric to (an open subset of) a direct product of the form $F \times N_1 \times \cdots \times N_\ell$ where $N_1, \cdots, N_\ell$ are closed, Einstein manifolds with positive scalar curvature and non-negative sectional curvature and $F$ is a non-compact, flat manifold.
By rescaling $g$, we may assume that the largest eigenvalue of $\ric|_p$ is equal to one. 
Let $v \in T_p M$ be a unit vector for which $\ric(v, v) = 1$; since $\ric|_p$ is diagonal with respect to the direct product decomposition, up to reordering the compact factors, we may assume, without loss of generality, that $v$ is a tangent vector to the $N_1$ factor.
Let $w \in T_p M$ be another unit vector which is tangent to the flat factor $F$ in this product decomposition, and consider the surface $\Sigma \subset M$ consisting of points of the form $\exp_p(tv + sw)$ for $t$ and $s$ close enough to zero.
We then know that this is a totally geodesic surface in $M$ for which $g|_\Sigma$ is flat.
    
Choose normal coordinates $(x^i)$ centred at $p$ on $M$, adapted to the local product structure of $(M, g)$, for which $\partial_1|_p = v$ and $\partial_2|_p = w$; denote by $t$ and $s$ the corresponding coordinates with respect to these vectors.
Then, $g|_\Sigma = dt^2 + ds^2$ and $r^2|_\Sigma = t^2 + s^2$, where $r$ denotes the distance to $p$ in $M$.
The asymptotics of $G$ near $p$ given by Theorem \ref{appendix-th: curvature asymptotics of greens function near the pole} imply that
\[
G|_\Sigma = r^{2 - n}|_\Sigma \left(1 + \frac{1}{12} \left(t^2 + \frac{\scal(p)}{n - 4} r^2|_\Sigma \right) + O_2 (r^3)\right).
\]
For simplicity, put $\alpha := \scal(p)/(n - 4)$.
As $b \equiv G^{1/(2 - n)}$, an argument analogous to that from Corollary \ref{appendix-cllr: sharp c2 asymptotics for s near singularity} gives us the asymptotics
\[
b^2|_\Sigma = r^2 \left(1 - \frac{1}{6(n - 2)} \left( t^2 + \alpha r^2 \right) \right) + O_2(r^5).
\]
Now, essentially, one expects that if $\alpha \ll 1$, then the difference between $b^2$ and $r^2$ is proportional to $t^2 r^2|_\Sigma$, which is neither convex nor concave, and this will tell us that $\hess{b^2}|_\Sigma$ has an eigenvalue strictly larger than two.

Together with the remark that $\Sigma$ is a totally geodesic surface in $M$ -- which gives us that $\hess{b^2}|_\Sigma = \hess^{\Sigma} {b^2}$, where $\hess^\Sigma$ denotes the Hessian along $\Sigma$ -- a direct computation implies that
\pagebreak
\allowdisplaybreaks
\begin{align} \label{appendix-alix: asymptotics of b squared near the pole along geodesic surface}
    \hess b^2|_{\Sigma} - 2 g|_\Sigma &= - \frac{1}{3(n - 2)}\begin{pmatrix}
        (1 + 2 \alpha) t^2 + 6 \alpha s^2 & (2 + 4 \alpha) t s \\
        (2 + 4 \alpha) t s & (6 + 6 \alpha) t^2 + (1 + 2 \alpha) s^2
    \end{pmatrix} + O(r^3) \nonumber \\
    &= - \frac{r^2}{3(n - 2)}\left( \frac{1}{r^2}\begin{pmatrix}
        t^2 & 2 s t \\
        2 s t & 6 t^2 + s^2
    \end{pmatrix} + O(\delta) + O(r) \right),
\end{align}
where the last equality follows by our assumption that $\alpha < \delta$.
Since
\begin{align*}
\tr\begin{pmatrix}
    t^2 & 2 t s \\
    2 t s & 6 t^2 + s^2
\end{pmatrix} = 7 t^2 + s^2 > 0, 
\quad
\det \begin{pmatrix}
    t^2 & 2 t s \\
    2 t s & 6 t^2 + s^2
\end{pmatrix}
= - 3 t^2 s^2 + 6 t^4 = 3 t^2 (3 t^2 - r^2)
\end{align*}
the above determinant is negative for $t \ll r$.
Thus if we fix a small positive ratio of $t$ to $r$, this matrix may be written as $r^2 A$, where $A$ is a fixed, universal matrix (independent of $M$ and $n$), which has a unit eigenvector $w$ with negative eigenvalue $-c$ for some $c > 0$.
It then follows that
\[
\hess b^2|_\Sigma (w, w) - 2 = \frac{r^2}{3 (n - 2)}(c + o_\delta (1) + O_{r} (r)).
\]
Hence, if $\delta > 0$ is universally sufficiently small and $r > 0$ is sufficiently small depending on the geometry of $M$, we obtain that $\hess b^2|_\Sigma (w,w) - 2 > 0$ which finishes the proof.
\end{proof}

As a corollary, we have the following general situation for when the Hessian inequality from \eqref{appendix-eq: supposed sharp hessian inequality for non-parabolic manifolds} fails for non-parabolic manifolds.

\begin{corollary} \label{appendix-cllr: failure of hessian inequality for products}
Let $N_1, \cdots, N_\ell$ be closed, Einstein manifolds with positive scalar curvature and non-negative sectional curvature.

Then, for $m \in \N$ large enough, the Hessian inequality \eqref{appendix-eq: supposed sharp hessian inequality for non-parabolic manifolds} fails for the manifold $\R^m \times N_1 \times \cdots \times N_\ell$.
\end{corollary}

\begin{proof}
If $n$ denotes the dimension of $\R^m \times N_1 \times \cdots \times N_\ell$, then the hypothesis of Theorem \ref{appendix-th: counter examples for non-compact case}  -- that $\delta (n - 4) \ric(v, v) > \scal(p)$ for some unit vector $v \in T_p M$ -- is satisfied if $m$ is large enough, as taking a product with a large-dimensional Euclidean space does not change the scalar curvature.
The desired result is now implied by the aforementioned theorem.
\end{proof}

Whilst the sharp Hessian inequality \eqref{appendix-eq: supposed sharp hessian inequality for non-parabolic manifolds} fails in general, we can prove that it holds for all Einstein, non-parabolic manifolds in this setting, \ie for all flat manifolds.
To see this, let $(M, g)$ be a complete, non-parabolic, flat manifold of dimension $n \geq 3$.
We employ the same notation from \eqref{appendix-eq: supposed sharp hessian inequality for non-parabolic manifolds}, and denote by $\lambda : M \setminus \{p\} \to \R$ the largest eigenvalue of $\hess{b^2}$.

Observe that our computations from Section \ref{sec: sharp hessian inequality} apply verbatim to this case as well (as they amount to setting $k = 0$ everywhere throughout Section \ref{sec: sharp hessian inequality}), and we may thus obtain that $\Delta \lambda - 2n \langle \nabla \log b, \nabla \lambda \rangle > 0$ in the barrier sense at every maximum point of $\lambda$ where it is strictly greater than two.
Moreover, by Theorem \ref{appendix-th: curvature asymptotics of greens function near the pole}, we have that $\lambda \to 2$ at $p$.
To conclude \eqref{appendix-eq: supposed sharp hessian inequality for non-parabolic manifolds} (using a similar contradiction argument as for the proof of Theorem \ref{intro-th: sharp hessian estimate}), we would need to ensure that there exist points in $M \setminus \{p\}$ where $\lambda$ is maximal; this amounts to determining the asymptotics of $\lambda$ at infinity.
Since $(M, g)$ is flat, this means that it is isometric to a quotient of the form $\R^n/\Gamma$, where $\Gamma$ is a discrete subgroup of isometries of $\R^n$.
Using the structure theory of such groups (cf.~{\cite[Section 5.4]{RatcliffeHyp}}), one may show that, if $\Gamma$ is not trivial, then $\hess{b^2} \to 0$ at infinity on $M$, \ie that $\lambda \to 0$ at infinity.

\bibliographystyle{alpha}
\bibliography{bibliography}

\end{document}